\documentclass{cmslatex}
\usepackage[paperwidth=7in, paperheight=10in, margin=.875in]{geometry}
 \usepackage[backref,colorlinks,linkcolor=red,anchorcolor=green,citecolor=blue]{hyperref}
\usepackage{amsfonts,amssymb}
\usepackage{amsmath}
\usepackage{graphicx}
\usepackage{cite}
\usepackage{enumerate}

\usepackage{graphicx}

\usepackage{latexsym}
\usepackage{bm}
\usepackage{amsmath}
\usepackage{amssymb}
\usepackage{color,tikz}
\usetikzlibrary{matrix}
\usepackage{cases}
\usetikzlibrary{shapes,arrows}
\usepackage{makecell}
\usepackage{multicol, multirow}
\usepackage{epsfig}
\usepackage{hyperref}
\usepackage{graphicx}
\usepackage{euscript}
\usepackage{epsfig}
\usepackage{subfigure} 
\usepackage{algorithm} 
\usepackage{algorithmicx}
\usepackage{algpseudocode} 
\hypersetup{hidelinks,
	colorlinks=true,
	allcolors=blue,
	pdfstartview=Fit,
	breaklinks=true}

\newcommand{\cut}{\ensuremath{\mathrm{cut}}}
\newcommand{\argmin}{\ensuremath{\mathrm{argmin\,}}}
\newcommand{\argmax}{\ensuremath{\mathrm{argmax\,}}}
\newcommand{\ah}{\ensuremath{h_{\mathrm{anti}}}}
\renewcommand{\vec}[1]{\mbox{\boldmath$#1$}}
\newcommand{\vol}{\ensuremath{\mathrm{vol}}}
\newcommand{\sign}{\ensuremath{\mathrm{sign}}}
\newcommand{\sgn}{\ensuremath{\mathrm{Sgn}}}
\newcommand{\norm}[1]{\Vert #1\Vert_\infty}
\newcommand{\median}{\ensuremath{\mathrm{median}}}
\newcommand{\set}[1]{\{ #1\}}

\newcommand{\blue}{\textcolor{black}}

\renewcommand{\theequation}{\arabic{section}.\arabic{equation}}

   \allowdisplaybreaks
\begin{document}
 \title{Continuous iterative algorithms for anti-Cheeger cut\thanks{Received date, and accepted date (The correct dates will be entered by the editor).}}


          \author{Sihong Shao\thanks{CAPT, LMAM and School of Mathematical Sciences,  Peking University, Beijing 100871, China, (sihong@math.pku.edu.cn).}
          \and Chuan Yang \thanks{School of Mathematical Sciences,  Peking University, Beijing 100871, China, (chuanyang@pku.edu.cn).}}

         \pagestyle{myheadings} \markboth{CONTINUOUS ITERATIVE ALGORITHMS FOR ANTI-CHEEGER CUT}{S. SHAO AND C. YANG} \maketitle

          \begin{abstract}
               As a judicious correspondence to the classical maxcut, 
               the anti-Cheeger cut has more balanced structure, 
               but few numerical results on it have been reported so far.  
               In this paper, we propose a continuous iterative algorithm (CIA) for the anti-Cheeger cut problem through fully using an equivalent continuous formulation. It does not need rounding at all and has advantages that all subproblems have explicit analytic solutions, the objective function values are monotonically updated and the iteration points converge to a local optimum in finite steps via an appropriate subgradient selection. It can also be easily combined with the maxcut iterations for breaking out of local optima and improving the solution quality thanks to the similarity between the anti-Cheeger cut problem and the maxcut problem. The performance of CIAs is fully demonstrated through numerical experiments on G-set from two aspects: one is on the solution quality where we find that the approximate solutions obtained by CIAs are of comparable quality to those by the multiple search operator heuristic method; the other is on the computational cost where we show that CIAs always run faster than the often-used continuous iterative algorithm based on the rank-two relaxation.
          \end{abstract}
\begin{keywords}  Anti-Cheeger cut; Maxcut; Iterative algorithm; Subgradient selection; Fractional programming
\end{keywords}

 \begin{AMS}  90C27; 05C85; 65K10; 90C26; 90C32
\end{AMS}

\section{Introduction}
\label{intro}

\begin{figure}
	\centering
	\includegraphics[scale=0.25]{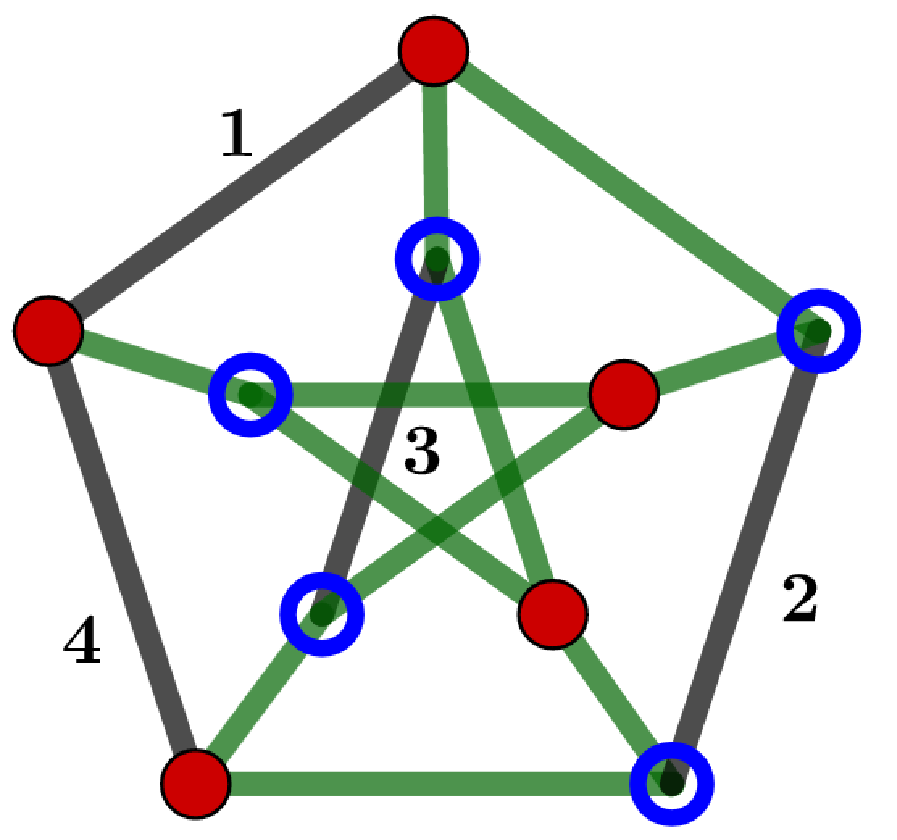}
	\quad
	\includegraphics[scale=0.25]{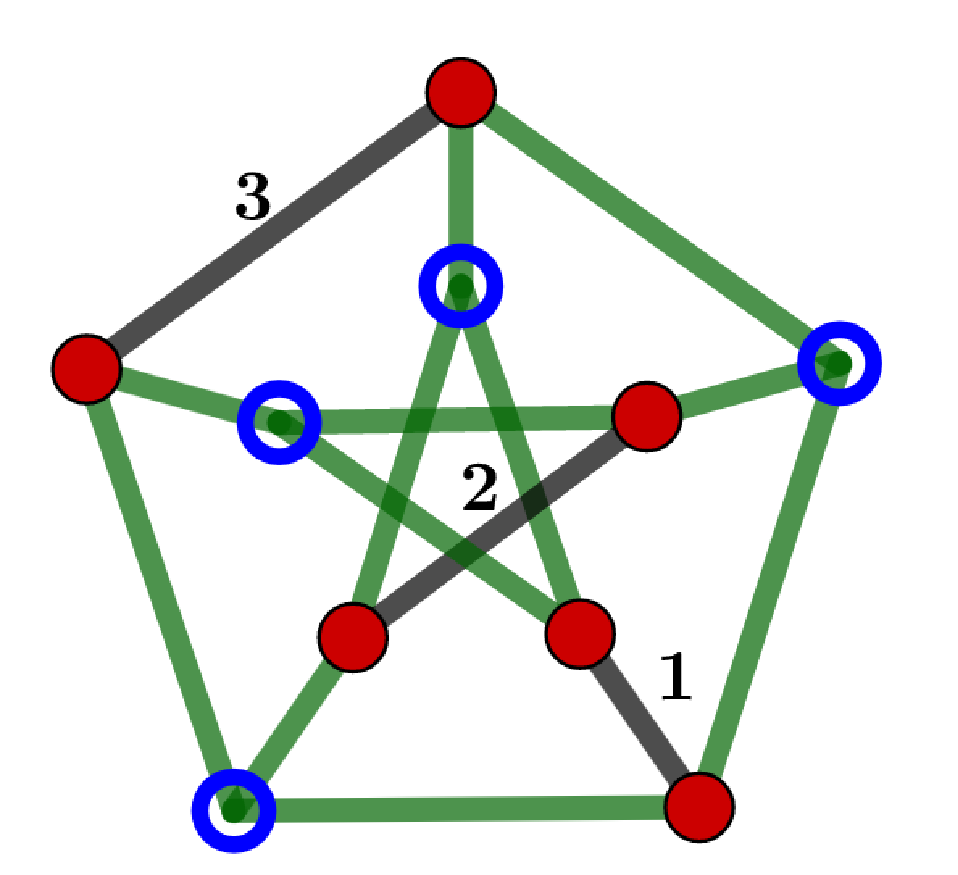}
	\caption{\small Anti-Cheeger cut (left) {\it v.s.} maxcut (right) on the Petersen graph: A bipartition $(S, S^\prime)$ of the vertex set $V$
		is displayed in red bullets ($S$) and blue circles ($S^\prime$).  For the anti-Cheeger cut, we have $\cut(S) = 11$,
		$\vol(S) = \vol(S^\prime)=15$, $\ah(G) = \frac{11}{15}$,
		and the remaining four uncut edges are divided equally in $S$ and $S^\prime$ ({\it judicious!}). 
		For the maxcut, we have $\cut(S) = 12$, $\vol(V) = 30$, $h_{\max}(G) = \frac{12}{15}$,
		but the remaining three uncut edges are all contained in $S$ ({\it biased!}).}
	\label{fig1}
\end{figure}

Given an undirected simple graph $G=(V,E)$ of order $n$ with the vertex
set $V$ and the edge set $E$,
a set pair $(S,S^\prime)$ with $S, S^\prime \not\in\{\emptyset,V\}$ is called a cut of $G$ if
$S\cap S^\prime = \varnothing$ and $S \cup S^\prime = V$ (i.e., $S^\prime = S^c$, the complementary set of $S$ in $V$),
and its cut value reads
\begin{equation}\label{eq:cut}
	\cut(S)=\sum\limits_{\{i,j\}\in E(S,S^\prime)}w_{ij},
\end{equation}
where $E(S,S^\prime)$ collects all edges cross between $S$ and $S^\prime$ in $E$, 
and  $w_{ij}$ denotes the nonnegative weight on the edge $\{i,j\}\in E$. 
In this work, we will focus on the following anti-Cheeger cut problem~\cite{Xu2016}:
\begin{equation} \label{eq:antiCheeger}
	\ah(G)=\max_{S\subset V}\frac{\cut(S)}{\max\{\vol(S),\vol(S^c)\}}.
\end{equation}
where $\vol(A) = \sum_{i\in A}d_i$ for $A\in\{S, S^c\}$ and $d_i=\sum_{\{i,j\}\in E}w_{ij}$.
Regarded as the corresponding judicious version, the anti-Cheeger cut~\eqref{eq:antiCheeger} may fix the biasness of the maxcut problem \cite{Karp1972,CSZZ2018}:
\begin{equation} \label{eq:maxcut}
	h_{\max}(G) = \max_{S\subset V}\frac{\cut(S)}{\frac12 \vol(V)},
\end{equation}
through an extra requirement that both $\vol(S)$ and $\vol(S^\prime)$ can not be too large.  
This can be readily observed on the Petersen graph as shown in Fig.~\ref{fig1},
and also implies that the anti-Cheeger cut may be harder than \blue{the} maxcut. 
Actually, the NP-hardness for the anti-Cheeger cut can be reduced from that for \blue{the} maxcut \cite{Karp1972}.
According to
$\vol(V) = \vol(S)+\vol(S^\prime)$ for any cut $(S,S^\prime)$ of $G$, we have 
\begin{equation} \label{eq:ah_max}
	\ah(G)\leq h_{\max}(G).
\end{equation}
In addition to the upper bound, it can be noted that a trivial lower bound is $\ah(G)\geq \frac{1}{2}$
which can be easily achieved by maximizing the cut value of any given cut in a greedy manner.  
Specifically, the resulting cut $(S,S^c)$ satisfies $|E(S,S)|\geq |E(S^c,S^c)|$; 
$\forall\, v\in S$, $|E(\{v\},S^c)|\geq |E(\{v\}, S)|$, and then 
summing over all the vertices in $S$ leads to $\cut(S)\geq \frac{1}{2} \max\{\vol(S),\vol(S^c)\}$.

Recently, it has been shown that  a set-pair  Lov\'asz extension produces some kinds of equivalent continuous optimization formulations for graph cut problems \cite{CSZZ2018}. The anti-Cheeger cut and \blue{the} maxcut~\eqref{eq:antiCheeger} can be respectively rewritten into
\begin{align} 
	\ah(G) &=\max_{\vec x\in\mathbb{R}^n\setminus \{\vec 0\}}\frac{I(\vec x)}{2\vol(V)\|\vec x\|_\infty - N(\vec x)}, \label{eq:antiCheeger-I(x)} \\
	h_{\max}(G) &=\max_{\vec x\in\mathbb{R}^n\setminus \{\vec 0\}}\frac{I(\vec x)}{\vol(V)\norm{\vec x}}, \label{eq:maxcut-continuous}
\end{align}
where
\begin{align}
	I(\vec x)&=\sum_{\{i,j\}\in E}w_{ij}|x_i-x_j|, \\
	N(\vec x)&=\min_{c\in\mathbb{R}}\sum_{i=1}^nd_i|x_i-c|,\\
	\|\vec x\|_\infty & =\max\{|x_1|,\ldots,|x_n|\}. \label{eq:inf}
\end{align}
Comparing Eq.~\eqref{eq:antiCheeger-I(x)}
with Eq.~\eqref{eq:maxcut-continuous}, 
it can be easily observed that the bias of \blue{the} maxcut is fixed by $N(\vec x)$ in the anti-Cheeger cut using 
the fact that 
\begin{equation}\label{fact}
	\alpha=\argmin_{c\in\mathbb{R}}\sum_{i=1}^n d_i|x_i-c|  ~~~ \text{ if and only if } ~~~ \alpha\in \median(\vec x),
\end{equation}
where $\median(\vec x)$ gives a range of the middle volume based on the order of $\vec x$ (see Definition 1.3 and Lemma 2.3 in \cite{Chang2015}). We can also derive from the above fact that $N(\vec x)\le \vol(V)||\vec x||_{\infty}$ and thus the objective function in Eq.~\eqref{eq:antiCheeger-I(x)} is positive. Actually,
if the equality holds in Eq.~\eqref{eq:ah_max}, then there exists at least one maxcut, denoted by ${S_*}$ in Eq.~ \eqref{eq:maxcut} or $\vec x_*$ in Eq.~\eqref{eq:maxcut-continuous},  satisfying $N(\vec x_*) = \vol(V)||\vec x_*||_{\infty}$ which gives nothing but the perfect balance: $\vol(S_*)=\vol(S_*^c)=\frac12\vol(V)$.  Just starting from the equivalent continuous formulation~\eqref{eq:maxcut-continuous}, a simple iterative (SI) algorithm was proposed for solving the maxcut problem \cite{SZZ2018}. The numerical maxcut solutions obtained by SI on G-set 
are of comparable quality to those by \blue{the} advanced heuristic combinational algorithms and 
have the best cut values among all existing continuous algorithms.
In view of the similarity between Eqs.~\eqref{eq:antiCheeger-I(x)} and \eqref{eq:maxcut-continuous},
it is then legitimate to ask whether we are able to get the same good performance when a similar idea to SI is applied into the former. We will give an affirmative answer in this work.

The first step to develop a continuous iterative algorithm (CIA) from the fractional form \eqref{eq:antiCheeger-I(x)} is to use the idea of Dinkelbach's iteration \cite{Dinkelbach1967},
with which we can equivalently convert the quotient into a difference of two convex functions. 
By linearizing the function to be subtracted (see $Q_r(\vec x)$ in Eq.~\eqref{eq:Qr}) via its subgradient, 
it is further relaxed to a convex subproblem,
the solution of which has an analytical expression.
That is, we are able to still have a kind of SI algorithm for the anti-Cheeger cut.
However, the selection of \blue{a} subgradient becomes more complicated than that for the maxcut
because $Q_r(\vec x)$ involves not only $I(\vec x)$ for calculating the cut values, which is also exactly the function to be subtracted in the maxcut problem~\eqref{eq:maxcut-continuous}, but also $N(\vec x)$ for measuring the judiciousness and the objective function value $r$.  
We cannot expect a satisfactory performance along with an arbitrary subgradient in $\partial Q_r(\vec x)$.
For instance, 
\textbf{CIA0}, our first iterative algorithm, constructs the subgradient in $\partial Q_r(\vec x)$
with another two subgradients, one chosen from $\partial I(\vec x)$ and the other $\partial N(\vec x)$, which are independent of each other, and produces numerical solutions with poor quality.
When the results obtained with the multiple search operator heuristic method (MOH) \cite{mahao2017} on G-set are set to be the reference,
the numerical lower bound of the ratios between the \blue{objective} function values achieved by \textbf{CIA0} and the reference ones are only about $0.870$. To improve it, a careful subgradient selection is then designed and the resulting iterative algorithm is named \textbf{CIA1}, which increases the above numerical lower bound to $0.942$. We also show that \textbf{CIA1} may get enough increase at each step (see Theorem~\ref{prop:S123}) and converges to a local optimum (see Theorem~\ref{thm:conver_3}). 

In order to make further improvement in solution quality by jumping out of local optima of \textbf{CIA1},  
we continue to exploit the similarity between Eqs.~\eqref{eq:antiCheeger-I(x)} and \eqref{eq:maxcut-continuous},
which shares a common numerator $I(\vec x)$ for the cut values. 
Specifically,  once \textbf{CIA1} cannot increase the objective function value for some successive iterative steps,
we switch to the SI iterations for the maxcut and denote the resulting algorithm by \textbf{CIA2}. Coupled with some 
additional simple random perturbations, numerical experiments show that the ratios between the \blue{objective} function values and the reference ones can be increased to at least {$0.994$}. Finally, in order to show the computational efficiency, we make a detailed comparison between CIAs and the often-used continuous iterative algorithm based on the rank-two relaxation (termed CirCut) \cite{burer2002rank},
and find that the former run much faster than the latter.

The paper is organized as follows. Section~\ref{sec:iter-scheme} presents three continuous iterative algorithms:
\textbf{CIA0}, \textbf{CIA1} and \textbf{CIA2}. Local convergence of \textbf{CIA1} is established in Section~\ref{sec:loc}.
Numerical experiments on G-set are performed in Section~\ref{sec:num}. \blue{Concluding remarks are given} in Section \ref{sec:conclusion}. 
It should be pointed out that both MOH and CirCut are originally designed for the maxcut problem~\eqref{eq:maxcut} and this work modifies them to generate approximate reference solutions for the anti-Cheeger cut problem~\eqref{eq:antiCheeger}. The descriptions for MOH and CirCut are presented in Appendix \ref{app:moh} and Appendix \ref{app:rank2}, respectively.

\section{Iterative algorithms}
\label{sec:iter-scheme}

Let $F(\vec x)$ denote the continuous, positive and fractional objective function of the anti-Cheeger cut problem~\eqref{eq:antiCheeger-I(x)}. Considering the zeroth order homogeneousness of $F(\vec x)$,  i.e., $F(\lambda\vec x) \equiv F(\vec x)$ holds for \blue{an} arbitrary $\lambda>0$,
it suffices to search for a maximizer of $F(\vec x)$ on a closed,  bounded and connected set $\Omega:=\{\vec x\in\mathbb{R}^n\big | \|\vec x\|_p= 1\}$, which can be readily achieved by the the following Dinkelbach iteration~\cite{Dinkelbach1967,SZZ2018}
\begin{subequations}
	\label{iter0}
	\begin{numcases}
		\vec x^{k+1}={\mathop{\argmin}\limits_{\|\vec x\|_p= 1} \{r^k \|\vec x\|_{\infty}-Q_{r^k}(\vec x)\}}, \label{iter0-1}\\
		r^{k+1}=F(\vec x^{k+1}),
	\end{numcases}
\end{subequations}
where 
\begin{align}
	Q_r(\vec x)&=\frac{I(\vec x)+rN(\vec x)}{2\vol(V)}, \quad r \ge 0, \label{eq:Qr}\\
	\|\vec x\|_p &= (|x_1|^p+|x_2|^p+\cdots+|x_n|^p)^{\frac{1}{p}}, \quad p\ge 1. 
\end{align}
Here we only need the intersection of $\Omega$ and any ray out from the original point 
is not empty and thus the choice of $p$ does not matter.


For any vector $\vec x^*$ reaching the maximum of $F(\vec x)$, there exists $\hat{\vec x}\in M(\vec x^*)$ such that 
\begin{equation}
	\frac{\hat{\vec x}}{||\hat{\vec x}||_{\infty}}=\vec 1_{S^*}-\vec 1_{(S^*)^c}
\end{equation}
also maximizes $F(\vec x)$ on $\vec x\in\mathbb{R}^n\setminus \{\vec 0\}$ thanks to the zeroth order homogeneousness of $F(\vec x)$,
where $S^*$ satisfies $S^+(\vec x^*) \subset S^* \subset (S^-(\vec x^*))^c$. Therefore, $(S^*,(S^*)^c)$ is an anti-Cheeger cut for graph $G$. Here $M(\vec x^*)=\{\vec x \big| ||\vec x||_{\infty}=||\vec x^*||_{\infty},S^\pm(\vec x^*)\subseteq S^\pm(\vec x)\}$, and $S^\pm(\vec x)=\{i\in V \big| x_i=\pm ||\vec x||_{\infty}\}$.

However, Eq.~\eqref{iter0-1} is still a non-convex problem and can be further relaxed to a convex one by using the fact that, 
for any convex and first degree homogeneous function $f(\vec x)$, we have 
\begin{equation}\label{eq:rex}
	f (\vec x)\geq  (\vec x, \vec s), \,\,\, \forall\, \vec s\in\partial f (\vec y), \,\,\, \forall\, \vec x, \vec y \in \mathbb{R}^n,
\end{equation}
where $(\cdot,\cdot)$ denotes the standard inner product in $\mathbb{R}^n$ and 
$\partial f (\vec y)$ gives the subgradient of function $f$ at $\vec y$.

Using the relaxation \eqref{eq:rex} in Eq.~\eqref{iter0-1} in view of that $Q_r(\vec x)$
is convex and homogeneous of degree one, 
we then modify the two-step Dinkelbach iterative scheme~\eqref{iter0} into 
the following three-step one
\begin{subequations}
	\label{iter1}
	\begin{numcases}{}
		\vec x^{k+1}=\mathop{\argmin}\limits_{\|\vec x\|_p= 1} \{r^k \|\vec x\|_{\infty} - (\vec x,\vec s^k)\}, 
		\label{eq:twostep_x2}
		\\
		r^{k+1}=F(\vec x^{k+1}),
		\label{eq:twostep_r2}
		\\
		\vec s^{k+1}\in\partial Q_{r^{k+1}}(\vec x^{k+1}),
		\label{eq:twostep_s2}
	\end{numcases}
\end{subequations}
with an initial data: 
$\vec x^0\in \mathbb{R}^n\setminus \{\vec0\}$, $r^0 = F(\vec x^0)$ and $\vec s^0\in\partial (Q_{r^0}(\vec x^0))$.

The iterative algorithm~\eqref{iter1} is called to be {\it simple} in \cite{SZZ2018} because the inner subproblem \eqref{eq:twostep_x2} has an analytical solution. That is, no any other optimization solver is needed in such simple three-step iteration. For the sake of completeness and the selection of \blue{a} subgradient in Eq.~\eqref{eq:twostep_s2}, we will write down the expression and some properties of the analytical solution and leave out all the proof details which can be found in \cite{SZZ2018}.

\begin{proposition}[Lemma 3.1 and Theorem 3.5 in \cite{SZZ2018}]
	\label{prop:rs}
	Suppose $\vec x^k$, $r^k$ and $\vec s^k$  are generated by the simple iteration \eqref{iter1}.
	Then for $k\geq 1$, we always have $0\le r^{k-1} \le r^k\leq \Vert \vec s^k\Vert_1$. In particular, 
	(1) $r^k= \Vert \vec s^k\Vert_1$ if and only if $\vec x^k/\norm{\vec x^k}\in \sgn(\vec s^k)$;
	and (2) $r^k<\Vert \vec s^k\Vert_1$ if and only if $L(r^k,\vec s^k)<0$. Here
	\begin{equation}\label{eq:L}
		{L}(r,\vec v) = \min_{\|\vec x\|_p= 1}\{ r \norm{\vec x} -(\vec x,\vec v)\}, 
		\,\,\,r\in\mathbb{R},\,\, \vec v\in\mathbb{R}^n, 
	\end{equation}
	denotes the minimal value of \eqref{eq:twostep_x2} and 
	$\sgn(\vec x) =\{(s_1, s_2,\ldots,s_n) \big| s_i\in \sgn(x_i),$ $i = 1,2,\ldots,n\}$
	with 
	\begin{equation}\label{sgnt}
		\sgn (t)=\begin{cases}
			\{1\}, &\;\text{if }t>0,\\
			\{-1\},&\;\text{if }t<0,\\
			[-1,1],&\;\text{if }t=0.
		\end{cases}
	\end{equation}
\end{proposition}

\begin{proposition}[Theorem 3.3 and Corollary 3.4 in \cite{SZZ2018}]
	\label{Thm:exact_solution}
	Let $\vec x^*$ be the solution of the minimization problem: 
	\[
	\vec x^*=\argmin_{\|\vec x\|_p= 1}\{ r \norm{\vec x} -(\vec x,\vec v)\}
	\]
	under the condition $0<r\le\|\vec v\|_1$.  Without loss of generality,  we assume  
	$|v_1|\geq|v_2|\geq\ldots\geq|v_n|\geq|v_{n+1}|=0$, and define $m_0=\min\{m\in\{1,2,\ldots,n\}:A(m)>r\}$
	and $m_1=\max\{m\in\{1,2,\ldots,n\}:A(m-1)<r\}$
	with $A(m)=\sum_{j=1}^m(|v_j|-|v_{m+1}|)$.
	\begin{description}
		\item[{\bf Scenario 1}] For $r<||\vec v||_1$ and $1<p<\infty$, we have $\vec x^*=\frac{\sign(\vec v)\vec z^*}{||\vec z^*||_p}$, 
		where $\vec z^*=(z_1,z_2,\ldots, z_n)$ with
		$z_i^*=\min\left\{1,\left(\frac{m_0|v_i|}{\sum_{j=1}^{m_0}|v_j|-r}\right)^{\frac{1}{p-1}}\right\}$.
		
		\item[{\bf Scenario 2}] For $r<||\vec v||_1$ and $p=1$, we have $\vec x^*=\frac{\sign(\vec v)\vec z^*}{||\vec z^*||_p}$, 
		where
		\[
		z_i^*=\left\{
		\begin{aligned}
			&1, &1\leq i\leq m_1,\\
			&0, &m_0\leq i\leq n,\\
			&[0,1], &m_1<i<m_0.
		\end{aligned}
		\right.
		\]
		\item[{\bf Scenario 3}] For $r=||\vec v||_1$ or $p=\infty$, the solution $\vec x^*$ satisfies 
		$\frac{\vec x^*}{||\vec x^*||_{\infty}}\in \sgn(\vec v)$ and $||\vec x^*||_p=1$.
	\end{description}
	Moreover, the minimizer $\vec x^*$ satisfies 
	\begin{equation}\label{eq:m}
		\frac{(\vec x^*,\vec v)}{\norm{\vec x^*}} \ge \sum_{i=1}^{m} |v_i|,
	\end{equation}
	where $m =m_0$ for  {\bf Scenario 1} and {\bf Scenario 2}, and $m=n$ for {\bf Scenario 3}. 
\end{proposition} 

The only thing remaining is to figure out the subgradient in Eq.~\eqref{eq:twostep_s2}.
Using the linear property of \blue{the} subgradient, we have 
\begin{equation}
	\label{eq:partial}
	\partial Q_{r}(\vec x)=\frac{1}{2\vol(V)}\left(\partial I(\vec x)+r\partial N(\vec x)\right),
\end{equation}
which means we may construct $\vec s\in \partial Q_{r}(\vec x)$ through $\vec u\in\partial I(\vec x)$ and $\vec v\in \partial N(\vec x)$ as
\begin{equation}\label{sg:gen}
	\vec s=\frac{1}{2\vol(V)}\left(\vec u+r\vec v\right).
\end{equation}

It was suggested in~\cite{SZZ2018} to use a point, denoted by $\bar{\vec p}  = (\bar p_1,\ldots,\bar p_n)$, located on the boundary of $\partial I(\vec x)$ as an indicator for selecting $\vec u\in\partial I(\vec x)$ to solve the maxcut problem \eqref{eq:maxcut-continuous}, where
\begin{align*}
	\bar{p}_{i} &=
	\left\{\begin{array}{ll}{p_{i} \mp q_{i},} & {\text { if } i \in S^{ \pm}(\vec x),} \\ {p_{i}+\operatorname{sign}\left(p_{i}\right) q_{i},} & {\text { if } i \in S^{<}(\vec x),}\end{array}\right. \\
	p_i &= \sum_{j:\{i,j\}\in E} w_{ij}\sign(x_i-x_j)-q_i, \quad 
	q_i = \sum_{j:\{i,j\}\in E\text{ and }x_i=x_j} w_{ij}, \\
	\sign (t) &=\begin{cases}
		1, &\;\text{if }t\ge0,\\
		-1,&\;\text{if }t<0,
	\end{cases}   \quad \text{and} \quad S^<(\vec x)=\{i\in V\big | |x_i|< ||\vec x||_{\infty}\}.
\end{align*}
It tries to select the boundary point at each dimension with the help of  a partial order relation `$\le$' on $\mathbb{R}^2$ by $(x_1,y_1)\le (x_2,y_2)$ if and only if either
$x_1<x_2$ or $x_1=x_2, y_1\le y_2$ holds. 
Given a vector $\vec a=(a_1,\ldots,a_n)\in\mathbb{R}^n$, let $\Sigma_{\vec a}(\vec x)$ be a collection of  permutations of $\set{1,2,\ldots,n}$ such that for any $\sigma\in\Sigma_{\vec a}(\vec x)$,  it holds 
\[
(x_{\sigma(1)},a_{\sigma(1)})\leq(x_{\sigma(2)},a_{\sigma(2)})\leq \cdots\leq (x_{\sigma(n)},a_{\sigma(n)}).
\]
Then for any $\sigma\in\Sigma_{\bar{\vec p}}(\vec x)$, 
we select 
\begin{equation}\label{sg:u}
	\vec u=(u_1,\ldots,u_n)\in\partial I(\vec x), \quad u_i = \sum_{j:\{i,j\}\in E}w_{ij} \sign(\sigma^{-1}(i)-\sigma^{-1}(j)).  \\
\end{equation}

It can be readily verified that $(\vec v, \vec 1)=0$ for any $\vec v\in \partial N(\vec x)$, 
and thus we may select
\begin{equation}\label{sg:v}
	\vec v = (v_1, \ldots, v_n)\in \partial N(\vec x), \quad v_i=\left\{\begin{array}{ll}{d_i\sign(x_i-\alpha),} & {\text { if } i\in V\backslash S^\alpha(\vec x),} \\ 
		{\frac{A}{B}d_i,} & {\text { if } i\in S^\alpha(\vec x),}\end{array}\right.
\end{equation}
where $\alpha\in \median(\vec x)$, $S^\alpha(\vec x)=\{i\in V \big |x_i=\alpha\}$, and
\[
A = -\sum_{i\in V\backslash S^\alpha(\vec x)}v_i=\sum_{x_i<\alpha}d_i-\sum_{x_i>\alpha}d_i, \quad B = \sum_{i\in S^\alpha(\vec x)}d_i
= \sum_{x_i=\alpha}d_i.
\]
Such selection of \blue{a} subgradient was already used to solve the Cheeger cut problem $h(G)=\min\limits_{S\subset V,S\not\in\{\varnothing,V\}} \frac{\cut(S)}{\min\{\vol(S),\vol(S^c)\}}$ in~\cite{Chang2015}
in view of the equivalent continuous formulation $ h(G)=\min\limits_{\vec x~~\mathrm{ nonconstant}} \frac{I(\vec x)}{N(\vec x)}$ \cite{CSZZ2018}. 

With those selections in Eqs.~\eqref{sg:u} and \eqref{sg:v} at our disposal,
we immediately obtain the first continuous iterative algorithm (abbreviated as \textbf{CIA0}) after using Eq.~\eqref{sg:gen} in the three-step iteration \eqref{iter1}. Although it still keeps the monotone increasing $r^{k+1}\ge r^k$ as stated in Proposition~\ref{prop:rs}, \textbf{CIA0} may not have enough increase at every iteration step because the selected $\vec s\in \partial Q_{r}(\vec x)$ is not located on the boundary of $\partial Q_{r}(\vec x)$. The numerical results in Table~\ref{tab:1} show clearly such unsatisfactory performance. Actually, 
according to Eq.~\eqref{eq:m}, the monotone increasing of $\{r^k\}_{k=1}^{\infty}$ has a direct connection to the subgradient $\vec s$ as follows
\begin{equation}\label{eq:rsr}
	r^{k+1}=\frac{Q_{r^{k+1}}(\vec x^{k+1})}{\norm{\vec x^{k+1}}}\ge \frac{Q_{r^k}(\vec x^{k+1})}{\norm{\vec x^{k+1}}}\ge \frac{(\vec x^{k+1},\vec s^k)}{\norm{\vec x^{k+1}}}\ge \sum_{i\in \iota} |s_i^k|\ge r^k,
\end{equation}
where there exists an index set $\iota\subset\{1,2,\ldots,n\}$, whose size should be $m$ in Eq.~\eqref{eq:m}.
In particular, such monotone increasing could be strict, i.e., $r^{k+1} > r^{k}$ via improving the last `$\ge$' in Eq.~\eqref{eq:rsr} into `$>$',  if $\Vert \vec s^k\Vert_1>r^k$  holds in each step.
To this end, we only need to determine a subgradient $\vec s^k \in \partial Q_{r^k}(\vec x^k)$ such that $\|\vec s^k\|_1>r^k$ if there exists $\vec s^k \in \partial Q_{r^k}(\vec x^k)$ such that $\|\vec s^k\|_1>r^k$. 
That is, we should hunt $\vec s^k$ on the boundary of $\partial Q_{r^k}(\vec x^k)$ after noting the convexity of the set $\partial Q_{r^k}(\vec x^k)$ and the function $\|\cdot\|_1$.

We know that, if $|S^{\alpha}(\vec x)|\leq 1$,  $\partial N(\vec x)$ contains just a single vector, denoted by $\vec a$, which is given by Eq.~\eqref{sg:v}, too; otherwise, $(\partial N(\vec x))_i$ becomes an interval for $i\in S^{\alpha}(\vec x)$, 
denoted by $[a_{i}^L, a_{i}^R]$ with $a_{i}^L=\max\{A-B+d_i,-d_i\}$ and $a_{i}^R=\min\{A+B-d_i,d_i\}$.
We still adopt a point $\vec b=(b_1,\ldots,b_n)\in\mathbb{R}^n$ on the boundary of $\partial Q_{r}(\vec x)$
as an indicator for choosing $\vec u\in\partial I(\vec x)$ and $\vec v\in \partial N(\vec x)$
such that the resulting $\vec s$ given by Eq.~\eqref{sg:gen} lives on the boundary of 
$\partial Q_{r}(\vec x)$, and the discussion is accordingly divided into two cases below.
\begin{description}
	\item[$\bullet~|S^{\alpha}(\vec x)|\le 1$:] $\,$\newline For $i=1,2,\ldots,n$, let 
	\begin{equation}\label{u1}
		b_{i}=\left\{\begin{array}{ll}{p_{i} \mp q_{i}+r a_i,} & {\text { if } i \in S^{ \pm}(\vec x),} \\ {p_{i}+r a_i+\operatorname{sign}\left(p_{i}+r a_i\right) q_{i},} & \text { if } i \in S^{<}(\vec x),
		\end{array}\right.
	\end{equation}
	and then we select 
	\begin{align}
		u_i &= \sum_{j:\{i,j\}\in E}w_{ij} \sign(\sigma^{-1}(i)-\sigma^{-1}(j)),  \quad \sigma\in\Sigma_{\vec b}(\vec x),  \label{c1:u}\\
		v_i &=a_i.\label{c1:v}
	\end{align}
	\item[$\bullet~|S^{\alpha}(\vec x)|\ge 2$:] $\,$\newline For $i=1,2,\ldots,n$, let 
	\begin{align}
		a_i&=
		\begin{cases}
			d_i\sign(x_i-\alpha),&\text{if $i\in V\setminus S^{\alpha}(\vec x)$},\\
			a^L_i,&\text{if $i\in S^{\alpha}(\vec x)\cap S^+(\vec x)$},\\
			a^R_i,&\text{if $i\in S^{\alpha}(\vec x)\cap S^-(\vec x)$},\\                                                                                                                                                                                                                                                                                                                                                                                                                                                                                                                                                                                                                                                                                                                                                                                                                                                                                                                                                                                                      \displaystyle \argmax_{t\in\{a^L_i,a^R_i\}}\{|p_i+rt|\} ,&\text{if $i\in S^{\alpha}(\vec x)\cap S^<(\vec x)$},
		\end{cases}
		\label{hh} \\
		b_{i}&=\left\{\begin{array}{ll}{p_{i} \mp q_{i}+r a_i,} & {\text { if } i \in S^{ \pm}(\vec x),} \\ {p_{i}+r a_i+\operatorname{sign}\left(p_{i}+r a_i\right) q_{i},} & {\text { if } i \in S^{<}(\vec x).}\end{array}\right. \label{u2}
	\end{align}
	Given $\sigma\in\Sigma_{\vec b}(\vec x)$ with the indicator $\vec b$ determined by Eq.~\eqref{u2}, we choose $\vec u\in\partial I(\vec x)$ according to Eq.~\eqref{c1:u}. 
	Let $i^*$ be the first index satisfying $x_{\sigma(i^*)}=\alpha$ and $i^*+t^*$ the last index satisfying $x_{\sigma(i^*+t^*)}=\alpha$, i.e.,  $x_{\sigma(i^*-1)}<x_{\sigma(i^*)}=x_{\sigma(i^*+1)}=\cdots=x_{\sigma(i^*+t^*)}=\alpha<x_{\sigma(i^*+t^*+1)}$, 
	and $j^*=\mathop{\argmax}\limits_{j\in\{\sigma(i^*),\sigma(i^*+t^*)\}}\{|b_j|\}$.
	Then, we select
	\begin{equation}
		\label{v2}
		v_i=
		\begin{cases}
			a_i,&\text{if }i\in V\backslash S^{\alpha}(\vec x),\\
			a_i,&\text{if }i\in S^{\alpha}(\vec x)\cap S^+(\vec x)\cap\{\sigma(i^*)\},\\
			\frac{A-a_{\sigma(i^*)}}{B-d_{\sigma(i^*)}}d_i,&\text{if }i\in S^{\alpha}(\vec x)\cap S^+(\vec x)\backslash\{\sigma(i^*)\},\\
			a_i,&\text{if }i\in S^{\alpha}(\vec x)\cap S^-(\vec x)\cap\{\sigma(i^*+t^*)\},\\
			\frac{A-a_{\sigma(i^*+t^*)}}{B-d_{\sigma(i^*+t^*)}}d_i,&\text{if }i\in S^{\alpha}(\vec x)\cap S^-(\vec x)\backslash\{\sigma(i^*+t^*)\},\\
			a_i,&\text{if }i\in S^{\alpha}(\vec x)\cap S^<(\vec x)\cap\{j^*\},\\
			\frac{A-a_{j^*}}{B-d_{j^*}}d_i,&\text{if }i\in S^{\alpha}(\vec x)\cap S^<(\vec x)\backslash\{j^*\}.\\
		\end{cases}
	\end{equation}
\end{description}


Plugging the above selection of \blue{a} subgradient into the three-step iteration \eqref{iter1}
yields the second continuous iterative algorithm, named as \textbf{CIA1} for short. 
It will be shown in next section that \textbf{CIA1} is sufficient to get enough increase at every step (see Theorem~\ref{prop:S123}) and
converges to a local \blue{optimum} (see Theorem~\ref{thm:conver_3}).

\begin{figure}
	\centering
	\tikzstyle{decision} = [diamond, draw, fill=green!20, text width=4em, aspect=1.2, text badly centered, node distance=3cm, inner sep=0pt]
	\tikzstyle{decision2} = [diamond, draw, fill=green!20, text width=6em, aspect=1.2, text badly centered, node distance=3cm, inner sep=0pt]
	\tikzstyle{block} = [rectangle, draw, fill=blue!20,
	text width=6em, text centered, rounded corners, minimum height=4em]
	\tikzstyle{io} = [rectangle, draw, fill=red!20,
	text width=4em, text centered, rounded corners, minimum height=2em]
	\tikzstyle{line} = [draw, very thick, color=black!50, -latex']
	\tikzstyle{cloud} = [draw, ellipse,fill=red!20, node distance=2.5cm,
	minimum height=2em]
	
	\begin{tikzpicture}[scale=2, node distance = 2cm, auto]
		\node [io] (init) {Start};
		\node [block, below of=init] (iterah) {Iterate anti-Cheeeger cut};
		\node [decision, below of=iterah,node distance=2.5cm] (ext1) {Exceed $T_{\text{tot}}$?};

		\node [io, below of= ext1, node distance=2cm](stop1){Stop};
		\node [decision2, right of=ext1, node distance=4cm] (keep_1) {Is $F_{\text{anti}}(\vec x)$  unchanged for $T_=$ steps?};
		\node [block, below of=stop1](itermc){Iterate maxcut};
		\node [decision, below of=itermc, node distance=2.5cm] (ext2) {Exceed $T_{\text{tot}}$?};
		\node [decision2, left of=itermc, node distance=4cm] (keep_2) {Is $F_{\max}(\vec x)$ unchanged for $T_=$ steps?};
		\node [io, below of=ext2, node distance=2cm] (stop) {Stop};
		\path [line] (init) -- (iterah);
		\path [line] (iterah) -- (ext1);
		\path [line] (ext1) -- node [, color=black] {Yes} (stop1);
		\path [line] (ext1) -- node [near start, color=black] {No} (keep_1);
		\path [line] (keep_1) |- node [anchor=south, color=black] {No} (iterah);
		\path [line] (keep_1) |- node [anchor=north, color=black] {Yes} (itermc);
		\path [line] (ext2) -| node [near start, color=black] {No} (keep_2);
		\path [line] (keep_2) |- node [anchor=east,color=black]{Yes}(iterah);
		\path [line] (keep_2) -- node [anchor=south,color=black]{No}(itermc);
		\path [line] (ext2) -- node [, color=black] {Yes}(stop);
		\path [line] (itermc) -- (ext2);
	\end{tikzpicture}
	\caption{\small The flowchart of \textbf{CIA2} --- a continuous iterative algorithm for the anti-Cheeger cut equipped with breaking out of local optima by the maxcut. $F_{\text{anti}}(\vec x)$ denotes the objective function 
		of the anti-Cheeger cut problem~\eqref{eq:antiCheeger-I(x)} and $F_{\max}(\vec x)$ the objective function of
		the maxcut problem~\eqref{eq:maxcut-continuous}. $T_{\text{tot}}$ gives the total iterative steps and $T_{=}$ counts the number of successive iterative steps in which the objective function values $F_{\text{anti}}(\vec x)$ or $F_{\max}(\vec x)$ keep unchanged.}
	\label{flowchart}
\end{figure}

In practice,  \textbf{CIA1} may converge fast to a local optimum (see e.g., Fig.~\ref{fig:smallsteps}), and thus need to cooperate with some local breakout techniques to further improve the solution quality. 
The similarity between the anti-Cheeger cut problem~\eqref{eq:antiCheeger-I(x)} and the maxcut problem~\eqref{eq:maxcut-continuous}, which shares a common numerator $I(\vec x)$ for the cut values,
provides a natural choice for us. We propose to switch to the SI iterations for the maxcut problem~\eqref{eq:maxcut-continuous} when \textbf{CIA1} cannot increase the objective function value for some successive iterative steps.
The resulting iterative algorithm is denoted by \textbf{CIA2}, the flowchart of which is displayed in Fig.~\ref{flowchart}.
It can be readily seen there that the maxcut iterations obtain an equal position with the anti-Cheeger cut iterations,
which means \textbf{CIA2} also improves the solution quality of the maxcut problem by using the anti-Cheeger cut to jump out of local optima at the same time. 
That is, the anti-Cheeger cut problem~\eqref{eq:antiCheeger-I(x)} and the maxcut problem~\eqref{eq:maxcut-continuous} are fully treated on equal terms by \textbf{CIA2}.

\section{Local convergence}
\label{sec:loc}

We already figure out the selection of \blue{a} subgradient (\textbf{CIA1}) in the iterative scheme~\eqref{iter1}, then we determine the selection of next solution $\vec x^{k+1}$ in Eq.~\eqref{eq:twostep_x2}. We choose the solution $\vec x^{k+1}\in \partial X_p^{k+1}$ in the same way as did in \cite{SZZ2018}, where $X_p^{k+1}$ is the closed set of all the solutions $\vec x^{k+1}$, and $\partial X_p^{k+1}$ is the boundary of $\partial X_p^{k+1}\in X_p^{k+1}$. Thus, the three-step iterative scheme~\eqref{eq:twostep_x2} of \textbf{CIA1} can be clarified as
\begin{subequations}
	\label{iter1c}
	\begin{numcases}{}
		\vec x^{k+1} \in \partial X_p^{k+1}, \label{eq:twostep_x2c}
		\\
		r^{k+1}=F(\vec x^{k+1}),
		\label{eq:twostep_r2c}
		\\
		\vec s^{k+1} = \vec s^\sigma, \,\,\, \sigma\in \Sigma_{\vec b}(\vec x^{k+1}).
		\label{eq:twostep_s2c}
	\end{numcases}
\end{subequations}
Let 
\begin{equation}\label{eq:C}
	C = \{\vec x\in\mathbb{R}^n \big| F(\vec y)\leq F(\vec x), \,\, \forall\,\vec y\in\{R_i\vec x: i\in\{1,\ldots,n\}\}\},
\end{equation}
where $R_i\vec x$ is defined as
\begin{equation}\label{eq:Tialp}
	(R_i\vec x)_j =
	\begin{cases}
		x_j, & j\ne i,\\
		-x_j, & j=i.
	\end{cases}
\end{equation}
The following two theorems verify two important properties of \textbf{CIA1} in~\eqref{iter1c}. Theorem~\ref{prop:S123} and Proposition~\ref{prop:rs} imply that the sequence $\{r^k\}$  strictly increases when there exists $\vec s \in \partial Q_{r^k}(\vec x^k)$ such that $\|\vec s\|_1>r^k$, {the opposite statement of which} is the necessary condition for the convergence of $\{r^k\}$. Theorem~\ref{thm:conver_3} guarantees the iterating scheme converges with $r^k=r^*$ and $\vec x^k=\vec x^*$ in finite steps, where $\vec x^*\in C$ is a local maximizer from the discrete point of view. Furthermore, we are able to prove that $\vec x^*$ is also a local maximizer in the neighborhood $U(\vec x^*)$: 
\begin{equation}
	\label{neighbor}
	U(\vec x^*)=\left\{\vec y\in \mathbb{R}^n\,\, \Big|\,\, \|\frac{\vec y}{\norm{\vec y}}-\frac{\vec x^*}{\norm{\vec x^*}}\|_\infty<\frac12\right\}.
\end{equation}

\begin{theorem}\label{prop:S123}
	We have $$\|\vec s^\sigma\|_1>r^k,\,\forall\, \sigma\in \Sigma_{\vec b}(\vec x^k),$$
	if and only if there exists $\vec s \in \partial Q_{r^k}(\vec x^k)$ satisfying $\|\vec s\|_1>r^k$.
\end{theorem}

\begin{proof} The necessity is obvious, we only need to prove the sufficiency. Suppose we can select a subgradient $\vec s=(s_1,\ldots,s_n)\in \partial Q_{r^k}(\vec x^k)$ satisfying $\|\vec s\|_1>r^k$, which directly implies $\vec x^k/ \norm{\vec x^k}\notin \sgn(\vec s)$ according to Proposition~\ref{prop:rs}. Accordingly, there exists an $i^*\in \set{1,2,\ldots,n}$ such that  
	$x_{i^*}^k/\norm{\vec x^k}\notin \sgn(s_{i^*})$, 
	and the corresponding component  $s_{i^*}$ satisfies
	\begin{equation}
		\left\{
		\begin{aligned}
			&s_{i^*}<0,&\text{if }i^*\in S^+(\vec x^k),\\
			&s_{i^*}>0,&\text{if }i^*\in S^-(\vec x^k),\\
			&|s_{i^*}|>0,&\text{if }i^*\in S^<(\vec x^k),\\
		\end{aligned}
		\right.
	\end{equation}
	which implies  
	\begin{equation}
		\label{eq:ee}
		\left\{
		\begin{aligned}
			&u_{i^*}\geq p_{i^*}-q_{i^*},\,v_{i^*}\geq a_{i^*}\,\Rightarrow\, 2\vol(V)s_{i^*}\geq b_{i^*},&\text{if }i^*\in S^+(\vec x^k),\\
			&u_{i^*}\leq p_{i^*}-q_{i^*},\,v_{i^*}\leq a_{i^*}\,\Rightarrow\, 2\vol(V)s_{i^*}\leq b_{i^*},&\text{if }i^*\in S^-(\vec x^k),\\
			&|u_{i^*}+r^kv_{i^*}|\leq |b_{i^*}|\,\Rightarrow\, 2\vol(V)|s_{i^*}|\leq |b_{i^*}|,&\text{if }i^*\in S^<(\vec x^k).\\
		\end{aligned}
		\right.
	\end{equation}
	
	Let $S^=_{i^*}(\vec x^k)=\{i\big | x_i^k=x_{i^*}^k\}$,
	and $i_1,i_2\in S^=_{i^*}(\vec x^k)$ be the indexes minimizing and maximizing function $\sigma^{-1}(\cdot)$ over $S^=_{i^*}(\vec x^k)$, respectively. 
	Then, we have 
	\begin{align*}
		s^\sigma_{i_1}=& \frac{1}{2\vol(V)}\left(\sum_{t:\{t,i_1\}\in E}w_{i_1t}\sign(\sigma^{-1}(i_1)-\sigma^{-1}(t))+r^kv_{i_1}\right)\\
		= &\frac{1}{2\vol(V)}\left(p_{i_1}-q_{i_1}+r^kv_{i_1}\right),\\
		s^\sigma_{i_2}=&\frac{1}{2\vol(V)}\left(\sum_{t:\{t,i_2\}\in E}w_{i_2t}\sign(\sigma^{-1}(i_2)-\sigma^{-1}(t))+r^kv_{i_2}\right)\\
		=& \frac{1}{2\vol(V)}\left(p_{i_2}+q_{i_2}+r^kv_{i_2}\right),
	\end{align*}
	{and the subgradient} selection in Eqs.~\eqref{u1}-\eqref{v2} yields  
	\begin{equation}
		\left\{
		\begin{aligned}
			&s_{i_1}^\sigma=b_{i_1}\leq b_{i^*}\leq 2\vol(V)s_{i^*}<0, &\text{if }i^*\in S^+(\vec x^k),\\
			&s_{i_2}^\sigma=b_{i_2}\geq b_{i^*}\geq 2\vol(V)s_{i^*}>0, &\text{if }i^*\in S^-(\vec x^k),\\
			&\max\{|s_{i_1}^\sigma|,|s_{i_2}^\sigma|\}=\max\{|b_{i_1}|,|b_{i_2}|\}\geq |b_{i^*}|>0, &\text{if }i^*\in S^<(\vec x^k),\\
		\end{aligned}
		\right.
	\end{equation}
	where we have used Eq.~\eqref{eq:ee}. 
	Therefore, there exists $i\in\{i_1,i_2\}$ satisfying 
	$x_{i}^k/\norm{\vec x^k}\notin \sgn(s_{i}^\sigma)$, 
	and thus we have $||\vec s^\sigma||_1>r^k$ according to Proposition~\ref{prop:rs}. 
\end{proof}

\begin{theorem}[finite-step local convergence]
	\label{thm:conver_3}
	Assume the sequences $\{\vec x^k\}$ and $\{r^k\}$ are generated by \textbf{CIA1} in Eq.~\eqref{iter1c}   from any initial point $\vec x^0\in \mathbb{R}^n\setminus \{\vec0\}$. There must exist $N\in\mathbb{Z}^+$ and $r^*\in\mathbb{R}$ such that, for any $k>N$, 
	$r^k=r^*$ and $\vec x^{k+1}\in C$ are local maximizers in $U(\vec x^{k+1})$. 
\end{theorem}

\begin{proof}
	We first prove that the sequence $\{r^k\}$ takes finite values. It is obvious that for $p=1$, $p=\infty$ and $m=n$(one case in Eq.~\eqref{eq:m}), we have 
	\begin{equation}\label{eq::3value}
		x^{k+1}_i/||\vec x^{k+1}||_{\infty}\in\{-1,0,1\},\quad \forall\, i=1,2,\ldots,n,
	\end{equation}
	due to $\vec x^{k+1}\in \partial X_p^{k+1}$,
	and then $r^{k+1} (= F(\vec x^{k+1}))$ takes a limited number of values. 
	In view of the fact that $\{r^k\}$ increases monotonically (see Proposition~\ref{prop:rs}), there exists $N\in\mathbb{Z}^+$ and $r^k\in\mathbb{R}$ must take a certain value for $k>N$. 
	That is, we only need to consider the remaining case for $1<p<\infty$ and $m<n$. 
	
	Suppose the contrary, we have that the subsequence $\{r^{k_j}\}_{j=1}^{\infty}\subset \{r^k\}_{k=1}^{\infty}$ increases strictly. According to Eq.~\eqref{eq:rsr}, there exists a permutation $\eta^{k_j}:\{1,2,\ldots,n\}\rightarrow \{1,2,\ldots,n\}$ satisfying $$|s^{k_j}_{\eta^{k_j}(1)}|\geq |s^{k_j}_{\eta^{k_j}(2)}|\geq \ldots\geq |s^{k_j}_{\eta^{k_j}(n)}|,$$ then we have 
	\begin{equation}
		\label{rrss}
		r^{k_{j+1}}-r^{k_j}\geq r^{k_{j}+1}-r^{k_j}\geq \sum_{i=m+1}^n|s^{k_j}_{\eta^{k_j}(i)}|\geq |s_{\eta^{k_j}(m+1)}^{k_j}|.
	\end{equation}
	If $s_{\eta^{k_j}(m+1)}^{k_j}=0$ holds for some $k_j$, then $\vec x^{k_j+1}$ satisfies Eq.~\eqref{eq::3value} derived from Scenario 1 of Proposition~\ref{Thm:exact_solution}. Meanwhile, we also have $r^{k_j}=||\vec s^{k_j}||_1$, which implies $r^k=r^{k_j}$ for $k>k_j$. This is a contradiction. Therefore, we have $|s_{\eta^{k_j}(m+1)}^{k_j}|>0,\,\forall\, j\in\mathbb{N}$ for the subsequence $ \{r^{k_j}\}_{j=1}^{\infty}$.  
	The range of $\sum_{i\in \iota}|s_i^k|$ is 
	\begin{equation*} 
		\left\{\sum_{i\in \iota}(a_i u_i+rh_iv_i)\Big|\vec u^\sigma=(u_1,\ldots,u_n),\,\vec v^\sigma=(v_1,\ldots,v_n),\,\forall\,\sigma\in \Sigma_{\vec b}(\vec x),\, \forall\, \vec x,\iota\right\},
	\end{equation*}
	which satisfies $\vec a,\vec h\in\{-1,1\}^n$ and $a_i u_i+h_iv_i\geq 0, \,\,\forall\, i=1,2,\ldots,n$.
	Obviously, the choice of $\vec u$, $\vec v$, $\vec a$, $\vec h$ and $\eta^{k_j}$ are finite. Thus, there exists a subsequence $\{r^{k_{j_t}}\}_{t=1}^\infty\subset \{r^{k_j}\}_{j=1}^\infty$ satisfying that for any $t_1,\,t_2\in\mathbb{Z}^+$, we have $\vec u^{k_{j_{t_1}}}=\vec u^{k_{j_{t_2}}}$, 
	$\vec v^{k_{j_{t_1}}}=\vec v^{k_{j_{t_2}}}$, $\vec a^{k_{j_{t_1}}}=\vec a^{k_{j_{t_2}}}$, $\vec h^{k_{j_{t_1}}}=\vec h^{k_{j_{t_2}}}$ and $\eta^{k_{j_{t_1}}}=\eta^{k_{j_{t_2}}}$. 
	Since $\{r^{k_{j_t}}\}_{t=1}^\infty$ increases monotonically, there exists $N_0\in\mathbb{Z}^+$, $\beta >0$ such that $|s_{\eta(m+1)}^{k_{j_{t}}}|>\beta$ for $t>N_0$, which contradicts to $r^{k_{j_{t}}}\leq 1,t\rightarrow \infty$ from Eq.~\eqref{rrss}. Thus there exist $N\in\mathbb{Z}^+$ and $r^*\in\mathbb{R}$ such that $r^k=r^*$ for any $k>N$, thereby implying that 
	\begin{align*}
		r^k &=\|\vec s^k\|_1, \\
		X_p^{k+1}&=\{\vec x\,\big|\,\vec x/ \|\vec x\|_\infty\in\sgn(\vec s^k),\,\, \|\vec x\|_p = 1\},
	\end{align*}
	derived from Proposition~\ref{prop:rs}. 
	That is, $\forall\,\vec x^{k+1}\in\partial X_p^{k+1}$, we have
	\begin{align}
		\vec x^{k+1}/ \|\vec x^{k+1}\|_\infty&\in\sgn(\vec s), \quad \forall\, s\in\partial Q_{r^k}(\vec x^{k+1}), 
		\label{eq:x1}
		\\
		S^<(\vec x^{k+1})&=\varnothing. \label{eq:x2}
	\end{align}
	%
	
	
	Now we prove $\vec x^{k+1}\in C$ and neglect the superscript $k+1$ hereafter for simplicity. 
	Suppose the contrary that there exists $i\in S^\pm(\vec x)$ satisfying 
	\begin{equation}\label{eq:ass}
		F(R_i\vec x)>F(\vec x) \quad \Leftrightarrow \quad Q_{r^k}(R_i\vec x)-Q_{r^k}(\vec x)>0
	\end{equation}
	because of $\|R_i\vec x\|_\infty=\|\vec x\|_\infty$. 
	
	On the other hand, we have
	\begin{equation}
		I(R_i\vec x)-I(\vec x)=\pm\sum\limits_{j:\{j,i\}\in E} 2w_{ij}x_j,\,\,i\in S^\pm(\vec x),
	\end{equation}
	and
	\begin{equation}
		\begin{aligned}
			N(R_i\vec x)-N(\vec x)=&-|\Delta-d_i x_i|+|\Delta+d_i x_i|\\
			=&\begin{cases}
				2d_ix_i,&\text{if }\Delta> d_i||\vec x||_{\infty},\\
				2\Delta\sign(x_i),&\text{if }\Delta \le d_i||\vec x||_{\infty},\\ 
				-2d_i x_i,&\text{if }\Delta< -d_i||\vec x||_{\infty},
			\end{cases}
		\end{aligned}
	\end{equation}
	from which, it can be derived that 
	\begin{equation}\label{eq:mmm}
		\begin{cases}
			\median(\vec x)=\median(R_i\vec x)=\{1\},&\text{if }\Delta> d_i||\vec x||_{\infty},\\
			x_i\in\median(\vec x),\,-x_i\in\median(R_i\vec x),&\text{if }\Delta\leq d_i||\vec x||_{\infty},\\
			\median(\vec x)=\median(R_i\vec x)=\{-1\},&\text{if }\Delta<-d_i||\vec x||_{\infty},
		\end{cases}
	\end{equation}
	where we have used the notation $\Delta:=\sum_{j=1}^nd_jx_j-d_ix_i$.
	
	Let $\vec u=(u_1,\ldots,u_n)\in\partial I(\vec x)$ with 
	$$
	u_i=\sum\limits_{j:\{j,i\}\in E} 2w_{ij}z_{ij},\,\,\,z_{ij} =-x_j/\norm{\vec x}\in \sgn(x_i-x_j),
	$$
	and $\vec v=(v_1,\ldots,v_n)\in\partial N(\vec x)$ with 
	$$
	v_i =\begin{cases}
		-d_i,&\text{if }\Delta> d_i||\vec x||_{\infty},\\
		-\Delta\sign(x_i),&\text{if }\Delta \le d_i||\vec x||_{\infty},\\ 
		
		d_i,&\text{if }\Delta< -d_i||\vec x||_{\infty}.
	\end{cases}
	$$
	
	Combining Eqs.~\eqref{sg:gen},  \eqref{eq:ass} and \eqref{eq:mmm} together 
	leads to 
	\begin{equation}
		\label{xisileq0}
		x_is_i=\frac{x_i u_i+r^kx_iv_i}{2\vol(V)}<0,
	\end{equation}
	which contradicts Eq.~\eqref{eq:x1}.


	Finally, we prove $\vec x$ is a local maximizer of $F(\cdot)$ over $U(\vec x)$.
	Let
	$$
	\vec y'=\frac{\norm{\vec x}}{\norm{\vec y}}  \vec y, \quad 
	g(t)=Q_{F(\vec x)}(t(\vec y'-\vec x)+\vec x), \quad \forall\,\vec y\in U(\vec x),
	$$
	and we claim that
	\begin{equation}\label{eq:g}
		g(t)-g(0)\leq 0, \quad \forall\,t\in[0,1],
	\end{equation}
	with which we are able to verify $\vec x$ is a local maximizer as follows 
	\begin{equation*}
		\begin{aligned}
			F(\vec y)-F(\vec x)=&F(\vec y')-F(\vec x)\\
			=&\frac{2\vol(V)}{2\vol(V)\|\vec y'\|_\infty-N(\vec y')}(Q_{F(\vec x)}(\vec y')-Q_{F(\vec x)}(\vec x))\\
			=&\frac{2\vol(V)}{2\vol(V)\|\vec y'\|_\infty-N(\vec y')}(g(1)-g(0))\le 0,
			\;\; \forall\,\vec y\in U(\vec x).
		\end{aligned}
	\end{equation*}

	Now we only need to prove Eq.~\eqref{eq:g}. It is obvious that $g(t)$ is linear on $[0,1]$ and $||\vec y^\prime||_{\infty}=||\vec x||_{\infty}$, and there exists $\vec s=(s_1,\ldots,s_n)$ satisfying $\vec s \in\partial Q_{F(\vec x)}(\vec x)$ and $\vec s \in\partial Q_{F(\vec x)}(\vec y^\prime)$, which implies that $(\vec s,\vec y'-\vec x)$
	is the slope of $g(t)-g(0)$, i.e., $g(t)-g(0)=t(\vec s,\vec y'-\vec x)$. With the fact that $\sign(\vec y^\prime)=\sign(\vec x)=\sign(\vec s)$, we have
	$$
	x_j (y'_j-x_j)\le 0\Rightarrow s_j(y'_j-x_j)\le 0, \quad  j=1,2,\ldots,n, 
	$$
	then the claim Eq.~\eqref{eq:g} is verified.
\end{proof}

\begin{remark}\rm
	\label{re:separate}
	We would like to point out that the subgradient selection in \textbf{CIA0} may not guarantee a similar local \blue{optimum}, which can be verified in two steps. For the first step, numerical experiments show that the solutions produced by \textbf{CIA0} on G-set
	are not contained in set $C$ with high probability.  For the second step, we can prove that there exists $\vec y\in U(\vec x)$ such that $F(\vec y)>F(\vec x)$ for any $\vec x\notin C$. The proof is briefed as follows. $\vec x\notin C$ implies that there exists $i\in S^\pm(\vec x)$ such that $F(R_i\vec x)>F(\vec x)$, i.e., $Q_{F(\vec x)}(R_i\vec x)>Q_{F(\vec x)}(\vec x)$. Thus, according to Eq.~\eqref{xisileq0}, there exists $\vec s\in \partial Q_{F(\vec x)}(\vec x)$ such that $x_is_i<0$. Let $\vec y=(y_1,\ldots,y_n)\in\mathbb{R}^n$ with $y_j=x_j$ if $j\neq i$ and $y_j=\frac{x_j}{4}$ otherwise. 
	It can be readily verified that $\vec y\in U(\vec x)$, and 
	\begin{equation*}
		\begin{aligned}
			F(\vec y)-F(\vec x)=&F(\vec y {\|\vec x\|_{\infty}}/{\|\vec y\|_{\infty}})-F(\vec x)\\
			=&\frac{2\vol(V)}{2\vol(V)\|\vec y'\|_\infty-N(\vec y')}(g(1)-g(0))\\
			=&\frac{2\vol(V)}{2\vol(V)\|\vec y'\|_\infty-N(\vec y')}\cdot\left(-\frac{3x_is_i}{4}\right)>0.
		\end{aligned}
	\end{equation*}
\end{remark}

\section{Numerical experiments}
\label{sec:num}

In this section, we conduct performance evaluation of the proposed CIAs on the graphs with positive weight in G-set\footnote{Downloaded from {\tt https://web.stanford.edu/~yyye/yyye/Gset/}} and always set the initial data $\vec x^0$ to be the maximal eigenvector of the graph Laplacian \cite{DelormePoljak1993,PoljakRendl1995}. The three bipartite graphs $G48$, $G49$, $G50$ will not be considered because their \blue{optimal} cuts can be achieved at the initial step. 
We first utilize the numerical solutions produced by MOH \cite{mahao2017} as the reference to check the quality of approximate solutions obtained by CIAs (see Section~\ref{sec:cia_ref}). Then, we carry out a fair comparison between CIAs and CirCut \cite{burer2002rank} to further investigate the computational cost (see Section \ref{sec:cia_cirah}). 
It should be pointed out that both MOH and CirCut are originally designed for the maxcut problem~\eqref{eq:maxcut} and we have to modify them to produce approximate reference solutions for the anti-Cheeger cut problem~\eqref{eq:antiCheeger} in this work. The interested readers may find more details on MOH and CirCut for the anti-Cheeger cut in Appendix \ref{app:moh} and Appendix \ref{app:rank2}, respectively.

%
%
%

\subsection{\blue{\textbf{Comparison with MOH in solution quality}}
\label{sec:cia_ref}}

\begin{table}
\centering
\caption{\small
	Numerical results for the anti-Cheeger problem by continuous iterative algorithms on G-set. 
	The \blue{objective} function values obtained with the multiple search operator heuristic method \cite{mahao2017} are chosen to be the reference (see the second column). The third, fourth and fifth columns present the \blue{objective} function values obtained with
	\textbf{CIA0}, \textbf{CIA1} and \textbf{CIA2}, respectively. 
	It can be easily seen that the numerical lower bound of the ratios between the \blue{objective} function values achieved by \textbf{CIA0} and the reference ones are only about $0.870$ (see G36); and \textbf{CIA1} improves it to $0.942$ (see G35), where we have chosen the maximum \blue{objective} function values among $100$ runs from the same initial data and set the total iteration steps $T_{\text{tot}}=100$ for each run. Combined with a population updating manner, \textbf{CIA2} further improves such numerical lower bound to {$0.990$} (see G37), where we have set the population size $L=20$,
	$T_{=}=3$, and $T_{\text{tot}}=10000$. The sixth and the last columns count the number of updating the population, denoted by $\#_P$.
	For instance, $\#_P=8$ for G37 means that \textbf{CIA2} runs $\#_P\times L\times T_{\text{tot}}=8\times 20 \times 10000 = 1600000$ iteration steps in total. \textbf{CIA2-0.1P} further improves the worst ratio to 0.994 (see G36), and achieves the reference results on G1$\sim$G5, G44, G46 and G47.
}  \label{tab:1}
\small
\setlength{\tabcolsep}{2.5mm}{
	\begin{tabular}{|c|c|c|c|cc|cc|}
		\hline
		&                             &                               &                               & \multicolumn{2}{c|}{}                                    & \multicolumn{2}{c|}{}                                    \\
		& \multirow{-2}{*}{reference} & \multirow{-2}{*}{\textbf{CIA0}}        & \multirow{-2}{*}{\textbf{CIA1}}        & \multicolumn{2}{c|}{\multirow{-2}{*}{\textbf{CIA2}}}              & \multicolumn{2}{c|}{\multirow{-2}{*}{\blue{\textbf{CIA2-0.1P}}}}             \\ \cline{2-8} 
		&                             &                               &                               &                               &                          &                               &                          \\
		\multirow{-4}{*}{graph} & \multirow{-2}{*}{value}     & \multirow{-2}{*}{value}       & \multirow{-2}{*}{value}       & \multirow{-2}{*}{value}       & \multirow{-2}{*}{$\#_P$} & \multirow{-2}{*}{value}       & \multirow{-2}{*}{$\#_P$} \\ \hline
		G1                      & 0.6062                      & 0.5764                        & 0.5967                        & 0.6031                        & 2                        & {0.6062} & 4                        \\
		G2                      & 0.6059                      & 0.5791                        & 0.5943                        & 0.6023                        & 2                        & {0.6059} & 3                        \\
		G3                      & 0.6060                      & 0.5716                        & 0.5988                        & 0.6049                        & 2                        & {0.6060} & 3                        \\
		G4                      & 0.6073                      & 0.5814                        & 0.5970                        & 0.6049                        & 3                        & {0.6073} & 4                        \\
		G5                      & 0.6065                      & 0.5836                        & 0.5994                        & 0.6050                        & 2                        & {0.6065} & 4                        \\
		G14                     & 0.6527                      & 0.5905                        & 0.6264                        & 0.6485                        & 4                        & 0.6507                        & 8                        \\
		G15                     & 0.6542                      & 0.5951                        & 0.6217                        & 0.6486                        & 5                        & 0.6517                        & 4                        \\
		G16                     & 0.6533                      & 0.5952                        & 0.6213                        & 0.6482                        & 5                        & 0.6499                        & 6                        \\
		G17                     & 0.6529                      & 0.5869                        & 0.6247                        & 0.6473                        & 7                        & 0.6499                        & 3                        \\
		G22                     & 0.6683                      & 0.6387                        & 0.6554                        & 0.6655                        & 3                        & 0.6682                        & 7                        \\
		G23                     & 0.6675                      & 0.6381                        & 0.6552                        & 0.6642                        & 3                        & 0.6666                        & 4                        \\
		G24                     & 0.6671                      & 0.6187                        & 0.6551                        & 0.6644                        & 4                        & 0.6667                        & 8                        \\
		G25                     & 0.6673                      & 0.6358                        & 0.6523                        & 0.6627                        & 3                        & 0.6670                        & 5                        \\
		G26                     & 0.6667                      & 0.6265                        & 0.6503                        & 0.6634                        & 6                        & 0.6661                        & 6                        \\
		G35                     & 0.6524                      & 0.5713                        & {\emph{0.6144}} & 0.6463                        & 12                       & 0.6487                        & 7                        \\
		G36                     & 0.6523                      & {\emph{0.5674}} & 0.6157                        & 0.6460                        & 10                       & {\emph{0.6483}} & 7                        \\
		G37                     & 0.6524                      & 0.5865                        & 0.6199                        & {\emph{0.6460}} & 8                        & 0.6486                        & 8                        \\
		G38                     & 0.6526                      & 0.5786                        & 0.6192                        & 0.6470                        & 6                        & 0.6496                        & 12                       \\
		G43                     & 0.6667                      & 0.6422                        & 0.6536                        & 0.6645                        & 3                        & 0.6665                        & 4                        \\
		G44                     & 0.6657                      & 0.6359                        & 0.6533                        & 0.6639                        & 2                        & {0.6657} & 4                        \\
		G45                     & 0.6661                      & 0.6352                        & 0.6529                        & 0.6623                        & 3                        & 0.6655                        & 4                        \\
		G46                     & 0.6655                      & 0.6371                        & 0.6530                        & 0.6617                        & 5                        & {0.6655} & 6                        \\
		G47                     & 0.6663                      & 0.6265                        & 0.6511                        & 0.6631                        & 6                        & {0.6663} & 4                        \\
		G51                     & 0.6511                      & 0.5956                        & 0.6248                        & 0.6467                        & 5                        & 0.6490                        & 5                        \\
		G52                     & 0.6508                      & 0.5935                        & 0.6205                        & 0.6471                        & 6                        & 0.6488                        & 7                        \\
		G53                     & 0.6510                      & 0.5915                        & 0.6264                        & 0.6459                        & 3                        & 0.6476                        & 3                        \\
		G54                     & 0.6511                      & 0.5866                        & 0.6279                        & 0.6471                        & 5                        & 0.6474                        & 4                        \\ \hline
	\end{tabular}
}
\end{table}

This section presents and analyzes the approximate solutions produced by CIAs on G-set.
We have tried $p=1,2,\infty$ in Eq.~\eqref{eq:twostep_x2}
and found that, as we also expected in presenting the Dinkelbach iteration~\eqref{iter0}, the results for different  $p$ $(=1,2,\infty)$ are comparable. Hence we only report the results for $p=1$ below. 
Table~\ref{tab:1} presents the \blue{objective} function values obtained with the proposed three kinds of CIAs.
In view of the unavoidable randomness in determining both $\vec x^{k+1}$ and $\vec s^{k+1}$,
we re-run \textbf{CIA0} (resp.~\textbf{CIA1}) 100 times from the same initial data, 
and the third (resp.~fourth) column of Table~\ref{tab:1} records the maximum \blue{objective} function values among these $100$ runs. We can see that the numerical lower bound of the ratios between the \blue{objective} function values achieved by \textbf{CIA0} and the reference ones are only about $0.870$ (see G36); and \textbf{CIA1} improves it to $0.942$ (see G35). That is, \textbf{CIA1} has much better subgradient selection than \textbf{CIA0}, which has been already pointed out in Remark~\ref{re:separate}. Also with such careful selection of the subgradient, \textbf{CIA1} has a nice property of local convergence within finite steps as stated in Theorem~\ref{thm:conver_3}. However, at the same time, it also means that \textbf{CIA1} may converge fast to local optima and thus easily gets stuck in improving the solution quality further. This is clearly shown in Fig.~\ref{fig:smallsteps} where we have plotted the history of the maximum objective function values chosen among 100 runs of \textbf{CIA1} against the iterative steps on six typical graphs. In fact, \textbf{CIA1} stops improving the solution quality after 53 steps for G1, 
30 for G14, 25 for G22, 30 for G35, 17 for G43 and 39 for G51. This also explains why we choose a small $T_{\text{tot}}=100$ 
for both \textbf{CIA0} and \textbf{CIA1}. 

According to Proposition \ref{prop:rs} and Theorem \ref{thm:conver_3}, \textbf{CIA1} increases monotonically and converges to a local optimum, thus it theoretically improves the results of \textbf{CIA0}, which can be validated numerically. We re-run \textbf{CIA0} and \textbf{CIA1} 20 times with the initial data to be the maximal eigenvector of the graph Laplacian for each graph in G-set. The solution generated in each run by \textbf{CIA0} is used as initial data and further optimized by \textbf{CIA1}. The procedure is denoted as \textbf{CIA0+CIA1}. We set the total iteration steps $T_{\text{tot}}=100$ in each run for all the three algorithms. \blue{ Fig. \ref{fig:compare0} shows the minimum, mean and maximum anti-Cheeger cut values obtained by \textbf{CIA0}, \textbf{CIA1} and \textbf{CIA0+CIA1}.  Among all $27\times 20=540$ results, the percentage of improvements is more than $98\%$. Moreover, \textbf{CIA1} outperforms \textbf{CIA0} in all graph instances, and its approximate solutions are of comparable quality to \textbf{CIA0+CIA1}.}

\begin{figure}
\centering
\includegraphics[scale=0.6]{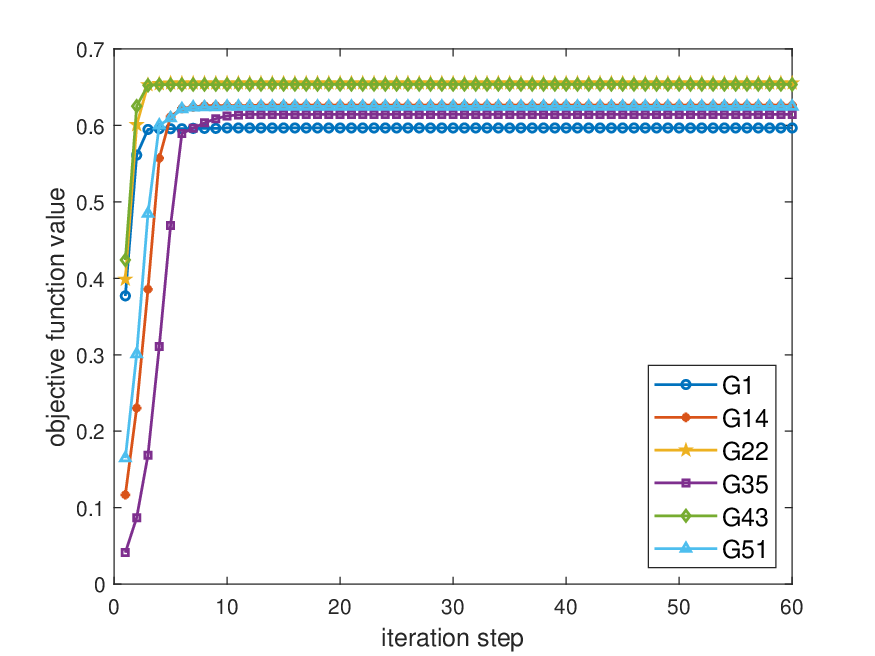}
\caption{\small The maximum objective function value chosen among 100 runs of \textbf{CIA1} from the same initial data {\it v.s.} iterative steps.  We can clearly observe that \textbf{CIA1} converges fast to a local optimum after a few iteration steps.  It stops increasing the maximum objective function values (see the fourth column of Table~\ref{tab:1}) after 53 steps for G1, 
	30 for G14, 25 for G22, 30 for G35, 17 for G43 and 39 for G51.}
\label{fig:smallsteps}
\end{figure}

\begin{figure}
\centering
\includegraphics[scale=0.3]{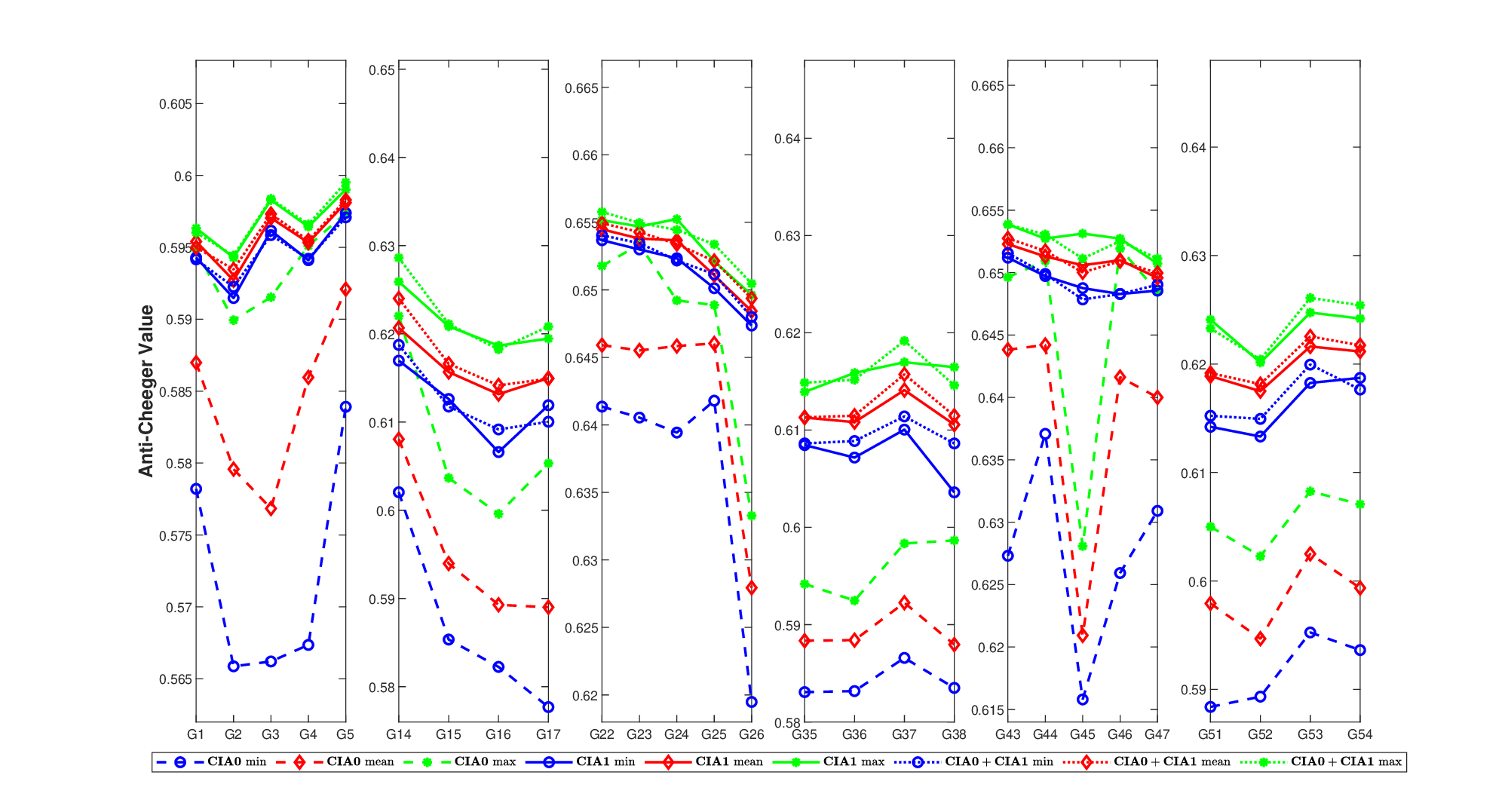}
\caption{\small \blue{ The minimum, mean, and maximum anti-Cheeger cut values produced by \textbf{CIA0}, \textbf{CIA1} and \textbf{CIA0+CIA1}. \textbf{CIA0+CIA1} means the output of \textbf{CIA0} serves as the input to \textbf{CIA1} for possible solution quality improvements. We find that, \textbf{CIA1} improves the results generated by \textbf{CIA0} in most cases, and its approximate solutions are of comparable quality to \textbf{CIA0+CIA1}.}}
\label{fig:compare0}
\end{figure}

Based on the similarity between the anti-Cheeger cut problem~\eqref{eq:antiCheeger-I(x)} and the maxcut problem~\eqref{eq:maxcut-continuous}, \textbf{CIA2}, the flowchart of which is given in Fig.~\ref{flowchart}, 
suggests a switch to the SI iterations for the latter when \textbf{CIA1} cannot increase the objective function values of the former within $T_=$ successive iterative steps. Such switch can be considered to be a kind of local breakout technique for further improving the solution quality. We have implemented \textbf{CIA2} in a population updating manner. 
In each round, the population contains $L$ numerical solutions, and its updating means precisely that $L$ independent \textbf{CIA2} runs from the same initial data occur simultaneously. We record the numerical solution which achieves the maximum objective function value among these $L$ runs,  adopt it as the initial data for the next round and define its objective function value to be the value of population. The population updating will not stop until the value of population keeps unchanged and we use $\#_P$ to count the number of such updating. 
After setting $L=20$ and $T_==3$, the ratios between the maximum \blue{objective} function values (i.e., the final values of population, see the fifth column of Table~\ref{tab:1}) and the reference ones can be increased by \textbf{CIA2} to at least {$0.990$} (see G37),
where a large $T_\text{tot}=10000$ is chosen to allow sufficient perturbations from local optima. 
In order to further improve the solution quality, we introduce some simple random perturbations in \textbf{CIA2}: randomly select and move $\gamma$ points from the original subset to the other subset, where $\gamma$ is used to control the perturbation strength. $\gamma$ is dynamically determined as randn($0.1|V|$, $0.3|V|$), where randn($0.1|V|$, $0.3|V|$) denotes a random integer between $0.1|V|$ and $0.3|V|$. The simple random perturbation is performed with probability 0.1 (\textbf{0.1P}) when SI or \textbf{CIA1} falls into a local optimum,
where the probability 0.1 reflects our numerical finding that applying such simple random perturbations infrequently with the population updating manner produces better approximate solutions. \blue{The resulting algorithm is denoted as \textbf{CIA2-0.1P}, and its numerical results are displayed in the last two columns of Table \ref{tab:1}.
We are able to see there that the minimum ratio between these values and the reference values is improved to $0.994$ (see G36).} In a word, Table~\ref{tab:1} shows evidently, on all 27 tested graphs,
the objective function values produced by \textbf{CIA2} are larger than those by \textbf{CIA1},
and the latter are larger than those by \textbf{CIA0} as well.
Moreover, the increased values from \textbf{CIA2} to \textbf{CIA1}, which are obtained with breaking out of local optima by the maxcut, are usually less than those from \textbf{CIA0} to \textbf{CIA1},
which are achieved with an appropriate subgradient selection.


Let's consider the computational time of CIAs. Besides some extra work on calculating $N(\vec x)$ and choosing
$\vec v\in\partial N(\vec x)$,  both of which can be easily achieved by sorting,  the detailed implementation of \textbf{CIA0} and \textbf{CIA1} is almost the same as that of SI for the maxcut problem,
and so does the cost. In this regard, an anti-Cheeger iteration and a maxcut iteration, both adopted in \textbf{CIA2} as shown in Fig.~\ref{flowchart}, are considered to have the same cost. 
The interested readers are referred to \cite{SZZ2018} for more details on the implementation and cost analysis, which are skipped here for simplicity.
Although \textbf{CIA0} and \textbf{CIA1} have the same cost, but the cost of the population implementation of \textbf{CIA2}
and \textbf{CIA2-0.1P} may be a little bit expensive
in some occasions. Taking G37 as an example, \textbf{CIA2} takes $\#_P\times L\times T_{\text{tot}} = 1.6\times 10^6$ iteration steps in total to reach $0.6460$ (the ratio is {$0.990$}) while \textbf{CIA2} only needs $100\times100=10^4$ to reach $0.6199$ (the ratio is $0.951$). That is, an increase by four percentage points in the ratio costs about a hundredfold price. However, a good news is the population updating for \textbf{CIA2} and \textbf{CIA2-0.1P} is a kind of embarrassingly parallel computation
and thus can be perfectly accelerated with the multithreading technology. \blue{The run times for all the algorithms in Table \ref{tab:1} are presented in Table \ref{tab:time1}.}


\begin{table}
\centering
\caption{\small
	\blue{The wall-clock time in seconds for CIAs in Table \ref{tab:1}. The time cost of \textbf{CIA0} or \textbf{CIA1} in each run is within 2 seconds. Due to the population updating manner, the run times of \textbf{CIA2} and \textbf{CIA2-0.1P}  are related to the scales of underlying graphs, and are much longer on the graph instances with 2000 vertices. }  
}  \label{tab:time1}
\small
\setlength{\tabcolsep}{2.5mm}{
	\begin{tabular}{|c|c|c|c|c|c|c|}
		\hline
		\multirow{4}{*}{graph} & \multirow{4}{*}{$|V|$} & \multirow{4}{*}{$|E|$} & \multirow{2}{*}{\textbf{CIA0}} & \multirow{2}{*}{\textbf{CIA1}} & \multirow{2}{*}{\textbf{CIA2}} & \multirow{2}{*}{\textbf{CIA2-0.1P}} \\
		&                        &                        &                       &                       &                       &                            \\ \cline{4-7} 
		&                        &                        & \multirow{2}{*}{time} & \multirow{2}{*}{time} & \multirow{2}{*}{time} & \multirow{2}{*}{time}      \\
		&                        &                        &                       &                       &                       &                            \\ \hline
		G1                     & 800                    & 19176                  & 0.58                  & 0.69                  & 106.56                & 189.60                     \\
		G2                     & 800                    & 19176                  & 0.49                  & 0.47                  & 100.83                & 143.23                     \\
		G3                     & 800                    & 19176                  & 0.48                  & 0.45                  & 176.69                & 145.50                     \\
		G4                     & 800                    & 19176                  & 0.47                  & 0.45                  & 97.36                 & 176.26                     \\
		G5                     & 800                    & 19176                  & 0.48                  & 0.46                  & 96.09                 & 177.31                     \\
		G14                    & 800                    & 4694                   & 0.26                  & 0.27                  & 192.73                & 375.77                     \\
		G15                    & 800                    & 4661                   & 0.26                  & 0.26                  & 237.43                & 197.67                     \\
		G16                    & 800                    & 4672                   & 0.27                  & 0.27                  & 235.83                & 284.68                     \\
		G17                    & 800                    & 4667                   & 0.26                  & 0.26                  & 189.36                & 152.18                     \\
		G22                    & 2000                   & 19990                  & 1.62                  & 1.28                  & 474.82                & 1096.16                    \\
		G23                    & 2000                   & 19990                  & 1.60                  & 1.25                  & 478.85                & 624.55                     \\
		G24                    & 2000                   & 19990                  & 1.63                  & 1.26                  & 794.32                & 1262.92                    \\
		G25                    & 2000                   & 19990                  & 1.65                  & 1.28                  & 789.00                & 796.85                     \\
		G26                    & 2000                   & 19990                  & 1.63                  & 1.27                  & 639.60                & 956.72                     \\
		G35                    & 2000                   & 11778                  & 1.63                  & 1.10                  & 1233.38               & 1249.14                    \\
		G36                    & 2000                   & 11766                  & 1.61                  & 1.10                  & 1047.16               & 1237.33                    \\
		G37                    & 2000                   & 11785                  & 1.64                  & 1.05                  & 1069.36               & 1410.71                    \\
		G38                    & 2000                   & 11779                  & 1.74                  & 1.12                  & 1390.01               & 2109.29                    \\
		G43                    & 1000                   & 9990                   & 0.45                  & 0.40                  & 126.94                & 239.96                     \\
		G44                    & 1000                   & 9990                   & 0.44                  & 0.41                  & 123.76                & 236.43                     \\
		G45                    & 1000                   & 9990                   & 0.45                  & 0.41                  & 236.62                & 236.95                     \\
		G46                    & 1000                   & 9990                   & 0.44                  & 0.40                  & 128.33                & 339.04                     \\
		G47                    & 1000                   & 9990                   & 0.45                  & 0.40                  & 196.62                & 236.70                     \\
		G51                    & 1000                   & 5909                   & 0.39                  & 0.35                  & 196.06                & 321.51                     \\
		G52                    & 1000                   & 5916                   & 0.39                  & 0.35                  & 315.76                & 437.26                     \\
		G53                    & 1000                   & 5914                   & 0.42                  & 0.38                  & 379.37                & 200.78                     \\
		G54                    & 1000                   & 5916                   & 0.39                  & 0.35                  & 136.69                & 254.24                     \\ \hline
	\end{tabular}
}
\end{table}

\subsection{\textbf{Comparison with CirCut in computational cost}}
\label{sec:cia_cirah}

In this section, we compare the performance of CirCut and \textbf{CIA2} without using the population updating manner 
in terms of solution quality and computational cost in two scenarios: with or without perturbation. Let CirCut0 be the version without random perturbation and angular representation and only include one time of line-search and Procedure-CUT (see Fig.~\ref{fig::circut}). We will compare CirCut0 with \textbf{CIA2} since the latter only includes the iterations of the maxcut and the anti-Cheeger cut. 
On the other hand, in accordance with the fact that CirCut is optimized with the random perturbation (see Eq.~\eqref{eq:rp-circut}), we also add simple random perturbations (\textbf{P}) into \textbf{CIA2}: Randomly select $\gamma$ $\in$ randn($0.1|V|$, $0.3|V|$) points, each of which is moved from its current subset to the other. When SI in \textbf{CIA2} (see Fig.~\ref{flowchart}) converges to a local optimum of the maxcut, the random perturbation should be applied. The resulting algorithm is denoted as \textbf{CIA2-P} and competes with CirCut. If CirCut stops with random perturbation for an average of $\bar{T_p}$ times, then each run of \textbf{CIA2-P} performs simple random perturbations for $\bar{T_p}$ times as well. 
We perform $20$ runs of repeated experiments for each graph as well as for each algorithm and set $N=10$, $T_{=}=3$.
Other required parameters for the line-search and random perturbation processes in CirCut are chosen in the same ways with the implementations in Github\footnote{see {\tt https://github.com/MQLib/MQLib/tree/master/src/heuristics/maxcut}}. We implement all above-mentioned algorithms in Matlab (r2019b) on the computing platform of 2*Intel Xeon E5-2650-v4 (2.2GHz, 30MB Cache, 9.6GT/s QPI Speed, 12 Cores, 24 Threads) with 128GB Memory.

%
%
%


\blue{Fig. \ref{fig::compare} plots the minimum, mean and maximum values ??of the approximate solutions obtained by 
\textbf{CIA2}, \textbf{CIA2-P}, CirCut0 and CirCut. We are able to observe there that \textbf{CIA2} is better than CirCut0 on all graphs except G15,  and the approximate solutions obtained by \textbf{CIA2-P}  are competitive with CirCut.  In particular, \textbf{CIA2-P} generates better results than CirCut on G2$\sim$G5, G22$\sim$G26, G43$\sim$G47, whose $|E|/|V|$ are higher.} The last column of Table \ref{tab:2} counts the average number of random perturbations performed by CirCut. \blue{Table \ref{tab:2} also provides the run times of all the algorithms, and verifies that \textbf{CIA2} and \textbf{CIA2-P} run much faster.} 

\begin{figure}[htbp]
\centering
\caption{\small \blue{The minimum, mean, and maximum anti-Cheeger cut values obtained by CirCut0 and \textbf{CIA2} are presented in (a), and others obtained by CirCut and \textbf{CIA2-P} are provided in (b). The population updating manner is not used here.}}
\label{fig::compare}
\subfigure[\textbf{CIA2} {\it vs} CirCut0.]{\includegraphics[scale=0.3]{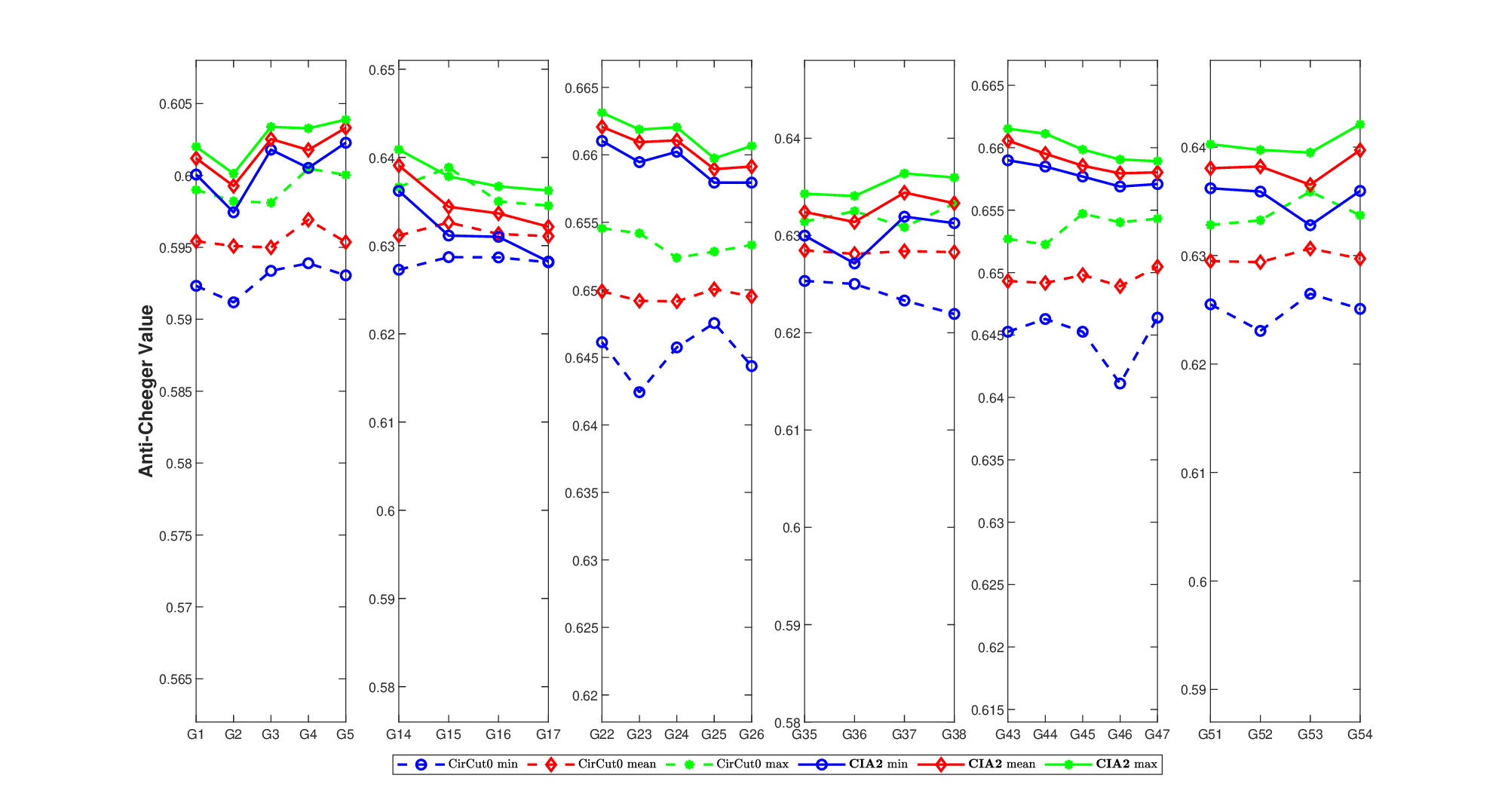}}
\subfigure[\textbf{CIA2-P} {\it vs} CirCut.]{\includegraphics[scale=0.3]{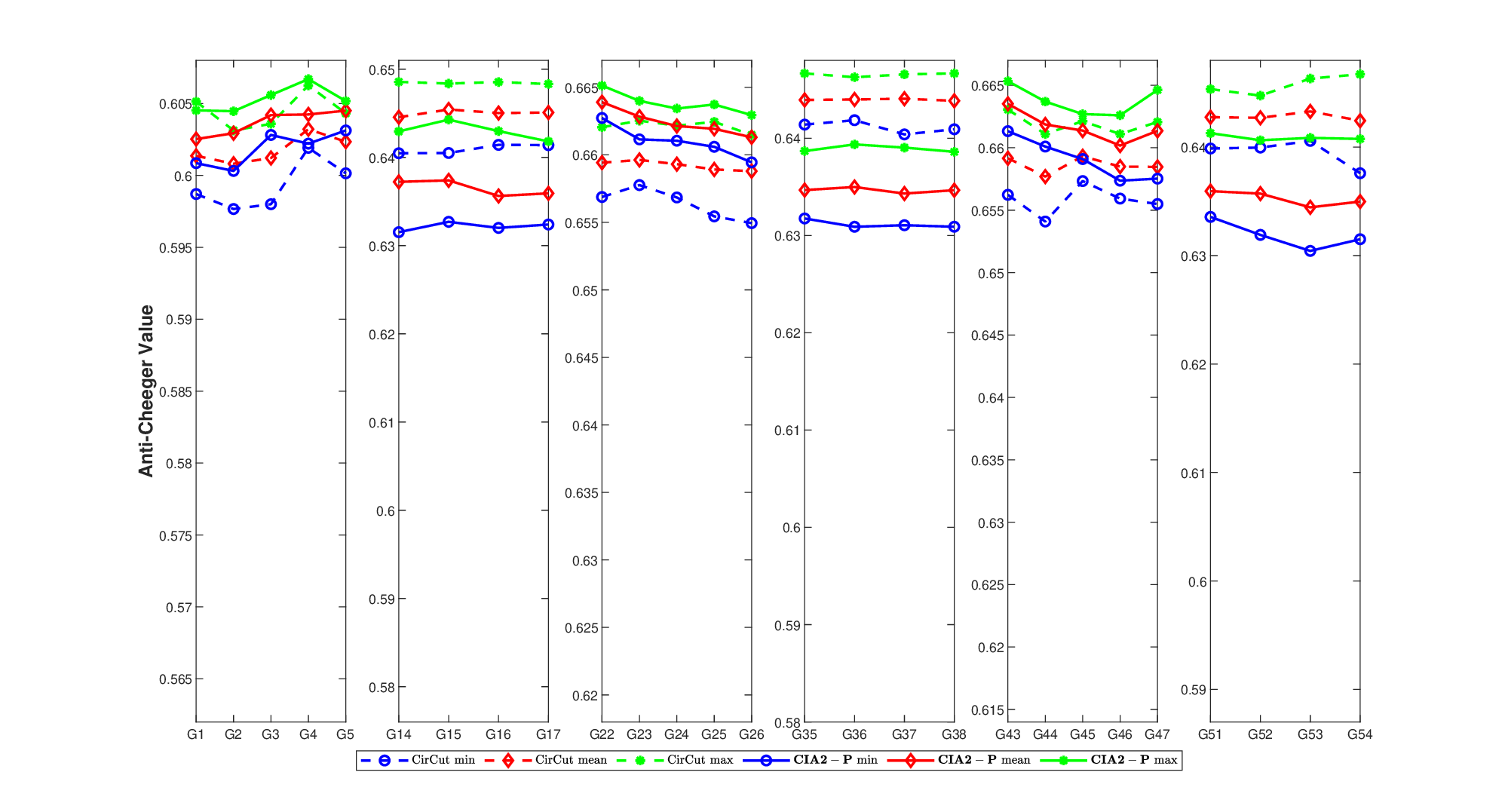}}
\end{figure}

\begin{table}[htbp]
\centering
\caption{\small \blue{The wall-clock time in seconds for \textbf{CIA2}, \textbf{CIA2-P}, CirCut0 and CirCut.
		The same number of perturbations $\bar{T}_P$ are used in CirCut and \textbf{CIA2-P}. It is clearly shown that \textbf{CIA2} and \textbf{CIA2-P} run much faster than CirCut0 and CirCut, respectively. The population updating manner is not used here.}}
\small     
\setlength{\tabcolsep}{2.5mm}{
	\begin{tabular}{|c|c|c|c|c|c|cc|}
		\hline
		\multirow{4}{*}{graph} & \multirow{4}{*}{$|V|$} & \multirow{4}{*}{$|E|$} & \multirow{2}{*}{\textbf{CIA2}} & \multirow{2}{*}{CirCut0} & \multirow{2}{*}{\textbf{CIA2-P}} & \multicolumn{2}{c|}{\multirow{2}{*}{CirCut}}   \\
		&                        &                        &                       &                         &                        & \multicolumn{2}{c|}{}                         \\ \cline{4-8} 
		&                        &                        & \multirow{2}{*}{time} & \multirow{2}{*}{time}   & \multirow{2}{*}{time}  & \multirow{2}{*}{time} & \multirow{2}{*}{$\bar{T}_p$} \\
		&                        &                        &                       &                         &                        &                       &                       \\ \hline
		G1                     & 800                    & 19176                  & 1.10                  & 2.12                    & 4.08                   & 44.80                 & 32                    \\
		G2                     & 800                    & 19176                  & 1.22                  & 2.09                    & 3.43                   & 40.46                 & 29                    \\
		G3                     & 800                    & 19176                  & 1.24                  & 2.07                    & 3.11                   & 34.04                 & 25                    \\
		G4                     & 800                    & 19176                  & 1.10                  & 2.48                    & 4.17                   & 37.77                 & 28                    \\
		G5                     & 800                    & 19176                  & 1.14                  & 2.15                    & 3.79                   & 39.44                 & 30                    \\
		G14                    & 800                    & 4694                   & 0.52                  & 1.35                    & 3.26                   & 23.02                 & 28                    \\
		G15                    & 800                    & 4661                   & 0.48                  & 1.36                    & 3.66                   & 26.65                 & 32                    \\
		G16                    & 800                    & 4672                   & 0.49                  & 1.48                    & 3.62                   & 26.06                 & 32                    \\
		G17                    & 800                    & 4667                   & 0.48                  & 1.44                    & 4.11                   & 27.78                 & 35                    \\
		G22                    & 2000                   & 19990                  & 3.32                  & 5.36                    & 16.42                  & 107.87                & 31                    \\
		G23                    & 2000                   & 19990                  & 3.32                  & 4.62                    & 16.58                  & 105.89                & 31                    \\
		G24                    & 2000                   & 19990                  & 3.33                  & 5.06                    & 15.69                  & 97.20                 & 29                    \\
		G25                    & 2000                   & 19990                  & 3.33                  & 5.07                    & 14.90                  & 94.60                 & 28                    \\
		G26                    & 2000                   & 19990                  & 3.44                  & 5.14                    & 18.43                  & 118.50                & 35                    \\
		G35                    & 2000                   & 11778                  & 2.61                  & 5.77                    & 20.24                  & 153.80                & 38                    \\
		G36                    & 2000                   & 11766                  & 2.59                  & 5.58                    & 25.77                  & 215.45                & 49                    \\
		G37                    & 2000                   & 11785                  & 2.53                  & 5.91                    & 21.23                  & 180.38                & 44                    \\
		G38                    & 2000                   & 11779                  & 2.67                  & 6.15                    & 19.53                  & 170.83                & 42                    \\
		G43                    & 1000                   & 9990                   & 0.93                  & 1.87                    & 4.76                   & 36.53                 & 33                    \\
		G44                    & 1000                   & 9990                   & 0.90                  & 1.86                    & 3.79                   & 27.95                 & 25                    \\
		G45                    & 1000                   & 9990                   & 0.97                  & 2.03                    & 4.56                   & 34.44                 & 31                    \\
		G46                    & 1000                   & 9990                   & 0.92                  & 1.88                    & 4.62                   & 34.07                 & 31                    \\
		G47                    & 1000                   & 9990                   & 0.92                  & 1.90                    & 4.15                   & 30.53                 & 28                    \\
		G51                    & 1000                   & 5909                   & 0.69                  & 2.21                    & 5.29                   & 38.21                 & 32                    \\
		G52                    & 1000                   & 5916                   & 0.70                  & 2.13                    & 5.03                   & 38.13                 & 31                    \\
		G53                    & 1000                   & 5914                   & 0.70                  & 2.33                    & 5.71                   & 41.30                 & 36                    \\
		G54                    & 1000                   & 5916                   & 0.68                  & 2.12                    & 5.51                   & 40.31                 & 35                    \\ \hline
	\end{tabular}
}
\label{tab:2}
\end{table}

\section{Conclusion and outlook}
\label{sec:conclusion}

Based on an equivalent continuous formulation, we proposed three continuous iterative algorithms (CIAs) for the anti-Cheeger cut problem, in which the \blue{objective} function values are monotonically updated and all the subproblems have explicit analytic solutions. With a careful subgradient selection, we were able to prove the iteration points converge to local optima in finite steps.  Combined with the maxcut iterations for
breaking out of local optima, the solution quality was further improved thanks to the similarity between the anti-Cheeger cut problem and the maxcut problem.
The numerical solutions obtained by our CIAs on G-set 
are of comparable quality to those by an advanced heuristic combinational algorithm. 
We will continue to explore the intriguing mathematical characters and useful relations in both graph cut problems, and use them for developing more efficient algorithms.

\par{\bf Acknowledgement.}
This research was supported by the National Key R~\&~D Program of China (No.~2022YFA1005102) and the National Natural Science Foundation of China (Nos.~12325112, 12288101, 11822102).
SS is partially supported by Beijing Academy of Artificial Intelligence (BAAI).
The authors would like to thank Dr.~Weixi Zhang for helpful discussions and useful comments.

\clearpage
\appendix
\section{Multiple search operator heuristic for the anti-Cheeger cut}
\label{app:moh}
The multiple search operator heuristic (MOH) method for the anti-Cheeger cut is based on MOH for the maxcut (see Algorithm 1 in \cite{mahao2017}), which is a state-of-the-art algorithm for the maxcut. MOH involves 5 search operators, consisting of local search, tabu search and perturbation. During the process of solving the maxcut, the variety of heuristic operators ensure that the neighborhood of a solution is widely explored, resulting in obtaining high-quality solution.

Since the original MOH method is specifically designed for $h_{\max}$, we make a slight modification for ensuring that the final solution approximates an anti-Cheeger cut. In the process of applying Algorithm 1 in \cite{mahao2017}, when the solution is updated by some search operator, we record the anti-Cheeger cut value synchronously. Suppose that MOH generates $M$ cuts: $(S_1,S_1^c)$, $(S_2,S_2^c)$, $\ldots$, $(S_M,S_M^c)$. We first select $(S_*,S_*^c)$ satisfying 
\[
(S_*,S_*^c)\in\mathop{\argmax}\limits_{(S,S^c)\in\left\{(S_i,S_i^c)\right\}_{i=1}^M}\left\{\frac{\cut(S)}{\max\{\vol(S),\vol(S^c)\}}\right\},
\] 
and then perform a greedy local search in the neighborhood of $(S_*,S_*^c)$ to make sure that the final solution is in $C$ (see Eq.~\eqref{eq:C}). It indicates that the essential idea of the modified MOH is to search for the anti-Cheeger cut in the neighborhood of the high-quality solutions for the maxcut.

The algorithm is implemented in C++ and compiled using GNU g++ with the optimization flag "-O2" on the same computing platform used by CirCut that mentioned in Section \ref{sec:cia_cirah}. We conduct 40 runs of repeated experiments on each graph instance. Each run is computed for 30 minutes, and the best anti-Cheeger cut value among the 40 results is considered the reference value. The numerical results are shown in Table \ref{tab:ref}.
According to Eq.~\eqref{eq:ah_max}, we may regard the best known maxcut value achieved by MOH as an experimental upper bound of $\ah(G)$,
and for each graph in G-set, this upper bound is extremely close to the reference value. The ratio between the reference values and the corresponding upper bounds is at least $0.9993$ (see G36), and eleven of them even reach exactly $1$. That is, in some sense,
the obtained reference values for $\ah$ on G-set are reliable.

\begin{table}[htbp]
	\centering
	\caption{\blue{Reference values and their equivalent fractional forms   for the anti-Cheeger cut by MOH on G-set (see the second column). For each graph instance $G$, the best known value for $h_{\max}(G)$ is shown in the third column, which is considered as the corresponding upper bound for $\ah(G)$. The ratio obtained by calculating $\ah(G)/h_{\max}(G)$ is represented in the last column.}}
	\label{tab:ref}
	\small
	\centering
	\setlength{\tabcolsep}{2mm}{
		\begin{tabular}{|c|c|c|c|}
			\hline
			\multirow{2}{*}{graph} & $\ah(G)$                    & $h_{\max}(G)$           & \multirow{2}{*}{ratio} \\
			& reference            & best known  &                        \\ \hline
			G1                     & 0.6062 (11624/19176) & 11624/19176 & 1                      \\
			G2                     & 0.6059 (11619/19177) & 11620/19176 & 0.999862               \\
			G3                     & 0.6060 (11620/19176) & 11622/19176 & 0.999828               \\
			G4                     & 0.6073 (11646/19178) & 11646/19176 & 0.999896               \\
			G5                     & 0.6065 (11630/19176) & 11631/19176 & 0.999914               \\
			G14                    & 0.6527 (3064/4694)   & 3064/4694   & 1                      \\
			G15                    & 0.6542 (3050/4662)   & 3050/4661   & 0.999785               \\
			G16                    & 0.6533 (3052/4672)   & 3052/4672   & 1                      \\
			G17                    & 0.6529 (3047/4667)   & 3047/4667   & 1                      \\
			G22                    & 0.6683 (13359/19991) & 13359/19990 & 0.99995                \\
			G23                    & 0.6675 (13343/19991) & 13344/19990 & 0.999875               \\
			G24                    & 0.6671 (13336/19990) & 13337/19990 & 0.999925               \\
			G25                    & 0.6673 (13340/19990) & 13340/19990 & 1                      \\
			G26                    & 0.6667 (13328/19990) & 13328/19990 & 1                      \\
			G35                    & 0.6524 (7685/11779)  & 7687/11778  & 0.999655               \\
			G36                    & 0.6523 (7676/11768)  & 7680/11766  & 0.999309               \\
			G37                    & 0.6524 (7689/11785)  & 7691/11785  & 0.99974                \\
			G38                    & 0.6526 (7687/11779)  & 7688/11779  & 0.99987                \\
			G43                    & 0.6667 (6660/9990)   & 6660/9990   & 1                      \\
			G44                    & 0.6657 (6650/9990)   & 6650/9990   & 1                      \\
			G45                    & 0.6661 (6654/9990)   & 6654/9990   & 1                      \\
			G46                    & 0.6655 (6649/9991)   & 6649/9990   & 0.9999                 \\
			G47                    & 0.6663 (6657/9991)   & 6657/9990   & 0.9999                 \\
			G51                    & 0.6511 (3848/5910)   & 3848/5909   & 0.999831               \\
			G52                    & 0.6508 (3851/5917)   & 3851/5916   & 0.999831               \\
			G53                    & 0.6510 (3850/5914)   & 3850/5914   & 1                      \\
			G54                    & 0.6511 (3852/5916)   & 3852/5916   & 1                      \\ \hline
		\end{tabular}
	}
\end{table}

\clearpage
\section{Rank-two relaxation for the anti-Cheeger cut}
\label{app:rank2}
CirCut is an efficient algorithm for solving the maxcut by rank-two relaxation \cite{burer2002rank}, which is a well known continuous framework for graph bipartitioning. Due to the high structural similarity between the anti-Cheeger cut and the maxcut, CirCut is expected to perform well on the anti-Cheeger cut. CirCut is an iterative algorithm (see Algorithm 1 in \cite{burer2002rank}), each round of which is described in Fig. \ref{fig::circut}. For the convenience of description, we use the notations in \cite{burer2002rank} unless otherwise specified. Starting  from an initial point $\vec \theta^0$, solve the rank-two relaxation 
\begin{equation}
	\label{eq:rank2-anti}
	\max_{\vec \theta}\quad h(\vec \theta)=\frac{\sum_{\{i, j\} \in E} w_{ij}\left(1-\cos \left(\theta_i-\theta_j\right)\right)}{\sum_{i=1}^n d_i+\sqrt{\sum_{i=0}^n \sum_{j=0}^n d_i d_j \cos \left(\theta_i-\theta_j\right)}}
\end{equation} 
for the anti-Cheeger cut. Suppose that we obtain a local optimum $\vec \theta$ for Eq. \eqref{eq:rank2-anti}, we apply Procedure-CUT (see Algorithm \ref{alg:procedure-cut}) to generate a binary cut $\vec x^*$, which serves as a rounding method for the anti-Cheeger cut. It is mentioned that simply displaying the objective function $F(\vec x)$ with $I(\vec x)$ (lines 4-6 in Algorithm \ref{alg:procedure-cut}) yields the rounding method for the maxcut (see Procedure-CUT in \cite{burer2002rank}). Set $\bar{\vec \theta}$ to be the angular representation of $\vec x^*$, namely
\begin{equation}
	{\bar{\theta}}_i=\left\{\begin{array}{cl}
		0, & \text { if } x^*_i=+1, \\
		\pi, & \text { if } x^*_i=-1.
	\end{array}\right.
\end{equation}
The random perturbation for $\bar{\vec \theta}$ is
\begin{equation}
	\label{eq:rp-circut}
	\bar{\vec \theta}+0.2\pi\operatorname{rand}[-1,1]^n,
\end{equation}
which derives the initial solution $\vec \theta^0$ for the next round.

Instead of applying the rank-two relaxation of the anti-Cheeger cut (see Eq. \eqref{eq:rank2-anti}), we find that the rank-two relaxation $\min f(\vec\theta)$ corresponding to the maxcut generates better results with 
\begin{equation}
	\label{eq:rank2-max}
	f(\vec \theta)=\sum_{\{i,j\}\in E}w_{ij}\cos(\theta_i-\theta_j).
\end{equation}
The cost of computing $h(\vec\theta)$ is $O(n^2)$, which causes trouble to the efficiency of computing. However, a simple backtracking Armijo line-search algorithm for solving $\min f(\vec\theta)$ (see Eq. \eqref{eq:rank2-max}) runs much faster. The similar idea is also applicable to the max bisection, which maximizes the cut value $I(\vec x)$ with the constraint $\sum_{i=1}^n x_i=0$. Both the anti-Cheeger cut and the max bisection can be regarded as balanced versions of the maxcut. It has been already pointed out in \cite{burer2002rank} that solving $\min f(\vec\theta)$ without the constraint $\vec e^T\cos(T(\vec \theta))\vec e=0$ obtains better bisections.

\begin{figure}[htbp]
	\centering
	\begin{tikzpicture}
		\matrix (m) [matrix of math nodes,row sep=3em,column sep=14em,minimum width=2em]
		{
			\vec \theta^0 & \vec \theta \\
			\bar{\vec \theta} & \vec x^* \\};
		\path[-stealth]
		(m-2-1) edge node [left] {Random perturbation} (m-1-1)
		(m-1-1.east|-m-1-2.west) edge node [above] {Rank-two relaxation solver} (m-1-2)
		(m-2-2.west|-m-2-1.east) edge node [below] {Angular representation} (m-2-1)
		(m-1-2) edge node [right] {Procedure-CUT} (m-2-2);
	\end{tikzpicture}
	\caption{Flow chart of CirCut per round for the anti-Cheeger cut.}
	\label{fig::circut}
\end{figure}

\begin{algorithm}
	\caption{\small Pseudo-code of Procedure-CUT. \textbf{Input:} $\vec \theta$. \textbf{Output:} final cut $\vec x^*$.}
	\label{alg:procedure-cut}
	\begin{algorithmic}[1]
		\State Sort $\theta_i\in[0,2\pi)$, $i=1,2,\ldots,n$ in the ascending order.
		\State Let $\alpha=0$, $i=1$ and $\Gamma=-\infty$. Let $j$ be the smallest index such that $\theta_j>\pi$ if there is one; otherwise set $j=n+1$. Set $\theta_{n+1}=2\pi$.
		\While{$\alpha\leq \pi$}
		\State Generate cut $\vec x$ by
		\begin{equation}
			x_i=\left\{\begin{array}{cc}
				+1, & \text { if } \theta_i \in[\alpha, \alpha+\pi), \\
				-1, & \text { otherwise, }
			\end{array}\right.
		\end{equation}
		and compute $F(\vec x)$.
		\If{$F(\vec x)>\Gamma$}
		\State let $\Gamma=F(\vec x)$, $\vec x^*=\vec x$.
		\EndIf
		\If{$\theta_i\leq \theta_j-\pi$}
		\State let $\alpha=\theta_i$ and increment $i$ by 1;
		\Else
		\State let $\alpha=\theta_j-\pi$ and increment $j$ by 1.
		\EndIf
		\EndWhile
	\end{algorithmic}
\end{algorithm}


%
%
%
%

          \end{document}